\documentclass[11pt,reqno]{amsart}
\usepackage[T1]{fontenc}
\usepackage[utf8]{inputenc}
\usepackage{amsmath,amssymb,amsfonts,amsthm,mathtools,booktabs,array}
\usepackage[a4paper,margin=2.5cm]{geometry}
\usepackage{lmodern}
\usepackage{microtype}
\usepackage{xcolor,needspace,float}
\usepackage{tikz}
\usetikzlibrary{arrows.meta}
\usepackage{hyperref}

\numberwithin{equation}{section}
\hypersetup{colorlinks=true,linkcolor=black,citecolor=black,urlcolor=black}
\allowdisplaybreaks
\newtheorem{theorem}{Theorem}[section]
\newtheorem{lemma}[theorem]{Lemma}
\newtheorem{proposition}[theorem]{Proposition}
\newtheorem{corollary}[theorem]{Corollary}
\newtheorem*{theoremA}{Theorem A}
\newtheorem*{theoremB}{Theorem B}
\newtheorem*{theoremC}{Theorem C}
\theoremstyle{definition}
\newtheorem{definition}[theorem]{Definition}

\theoremstyle{remark}
\newtheorem{remark}[theorem]{Remark}

\newcommand{\rev}[1]{{#1}}
\newcommand{\attribrev}[1]{{#1}}

\newcommand{\lastrev}[1]{{#1}}
\newcommand{\introrev}[1]{{#1}}
\newcommand{\func}[1]{\operatorname{#1}}

\begin{document}
\title[Polynomial BH Bounds and HL Inequalities for Finite Local Systems]{Polynomial Bohnenblust--Hille Bounds and Hardy--Littlewood Inequalities for Finite Local Systems}
\author[A. Barbosa et al.]{Anderson Barbosa}
\address{Universidade Federal Rural do Semi-\'Arido (UFERSA), Rio Grande do
Norte, Brazil}
\email{anderson.barbosa@ufersa.edu.br}
\author[]{Daniel N\'u\~nez-Alarc\'on}
\address{Department of Mathematics, Universidad Nacional de Colombia,
Bogot\'a, Colombia}
\email{Dnuneza@unal.edu.co}
\author[]{Daniel Pellegrino}
\address{Departamento de Matem\'atica, Universidade Federal da Para\'iba,
Jo\~ao Pessoa, PB, Brazil}
\email{daniel.pellegrino@academico.ufpb.br}
\author[]{Anselmo Raposo Jr.}
\address{Coordena\c{c}\~{a}o do Curso de Matem\'atica -- Bacharelado, Universidade Federal do Maranh\~ao,
S\~ao Lu\'is, MA, Brazil}
\email{anselmo.junior@ufma.br}
\date{}

\begin{abstract}
\rev{The complex polynomial Bohnenblust--Hille inequality states that, for a
$d$-homogeneous polynomial $P:\mathbb C^N\to\mathbb C$,
$P(z)=\sum_{|\alpha|=d}c_\alpha z^\alpha$, the coefficient norm
$(\sum_{|\alpha|=d}|c_\alpha|^{p_d})^{1/p_d}$ is bounded by
$B_d^{\mathrm{pol}}\|P\|_\infty$, where $p_d=2d/(d+1)$ and
$B_d^{\mathrm{pol}}$ is independent of $N$. Here $\|P\|_\infty$ is the
supremum norm on the unit polydisc.
This inequality has a Boolean analogue for the Fourier--Walsh coefficients
of functions of degree at most $d$, with the supremum taken over
$\{-1,1\}^N$. In this setting, Defant, Masty{\l}o and P\'erez obtained
subexponential bounds, and Ivanisvili established polynomial growth. Further polynomial
estimates were obtained by Slote, Tseng and Volberg in the Boolean setting
and by Slote and Volberg for cyclic groups.
We study these coefficient inequalities in a broader framework of finite
systems of functions that encompasses and extends the Boolean and cyclic
settings. This framework also admits a natural $\ell_p$ formulation,
leading to Hardy--Littlewood-type inequalities.
For orthonormal systems, we obtain polynomial bounds for the
Bohnenblust--Hille constants and, for sufficiently large $p$, for the
constants in these Hardy--Littlewood-type inequalities. In the Boolean
case, the polynomial power is $5/2$, with a cubic logarithmic factor.}
\end{abstract}

\maketitle
\tableofcontents

\section{Introduction}

The Bohnenblust--Hille inequality originated in the study of absolute
convergence of Dirichlet series. In 1931, Bohnenblust and Hille \cite%
{BohnenblustHille} determined the maximal width $1/2$ of the strip between
uniform and absolute convergence. The coefficient inequality underlying that
theorem has since become a basic tool in the study of homogeneous
polynomials, Dirichlet series and Bohr radii.

Let $d,N$ be positive integers, and put $p_{d}=2d/(d+1)$. For a multi-index $%
\alpha \in \mathbb{N}_{0}^{N}$, where $\mathbb{N}_{0}=\{0,1,2,\ldots \}$,
write $|\alpha |=\alpha _{1}+\cdots +\alpha _{N}$ and, for $z=\left(
z_{1},\ldots ,z_{N}\right) \in \mathbb{C}^{N}$, denote $z^{\alpha
}=\prod_{j=1}^{N}z_{j}^{\alpha _{j}}$. A complex $d$-homogeneous polynomial
is a map 
\begin{equation*}
P:\mathbb{C}^{N}\longrightarrow \mathbb{C},\qquad P(z)=\sum_{|\alpha
|=d}c_{\alpha }z^{\alpha }.
\end{equation*}%
\Needspace{4\baselineskip}
Writing $\mathbb{T}=\{z\in \mathbb{C}:|z|=1\}$, its supremum norm on the
polytorus is 
\begin{equation*}
\Vert P\Vert _{\mathbb{T}^{N}}:=\max_{z\in \mathbb{T}^{N}}|P(z)|.
\end{equation*}%
The polynomial Bohnenblust--Hille inequality asserts that 
\begin{equation}
\left( \sum_{|\alpha |=d}|c_{\alpha }|^{p_{d}}\right) ^{1/p_{d}}\leq B_{d}^{%
\mathrm{pol}}\Vert P\Vert _{\mathbb{T}^{N}},  \label{intro:classical-BH}
\end{equation}%
where $B_{d}^{\mathrm{pol}}$ depends on $d$ and is independent of $N$. The
exponent $p_{d}$ is optimal among dimension-free coefficient exponents.

The dependence on $d$ is essential in applications in which the degree grows
together with the dimension. Defant, Frerick, Ortega-Cerd\`a, Ouna\"ies and
Seip \cite{DefantFrerickOrtegaOunaiesSeip} proved that the polynomial
constants are hypercontractive: 
\begin{equation*}
B_d^{\mathrm{pol}}\le C^d 
\end{equation*}
for an absolute $C>1$. The subsequent subexponential estimates in \cite%
{BayartPellegrinoSeoane} were used to determine the asymptotic Bohr radius
of the polydisk, showing how quantitative coefficient bounds enter a
concrete problem of holomorphic analysis.
Polynomial growth of $B_d^{\mathrm{pol}}$ was subsequently established in
\cite[Theorem~A]{PellegrinoTeixeiraBH}.
On the Boolean cube, Defant, Masty{\l}o and P\'erez
\cite{DefantMastyloPerez} obtained subexponential bounds for the
Fourier--Walsh coefficients of real-valued functions. The same growth scale
holds for complex-valued functions by separating real and imaginary parts.
We write $\mathsf B_d^{\mathrm{Bool}}$ for the optimal constant for
complex-valued functions of degree at most $d$. More recently, Ivanisvili
\cite[Theorem~1.1]{Ivanisvili2026} proved
\begin{equation*}
\mathsf B_d^{\mathrm{Bool}}\le C d^{27}.
\end{equation*}
A sharper polynomial power was subsequently obtained by Slote, Tseng and
Volberg \cite[Theorem~2.1]{SloteTsengVolberg}: for every
$\varepsilon>0$,
\begin{equation}
\mathsf B_d^{\mathrm{Bool}}
\le
C_\varepsilon
d^{\frac32+\frac1{\log 2}+\varepsilon}
=
C_\varepsilon
d^{2.942695\ldots+\varepsilon}.
\label{eq:intro-STV}
\end{equation}
All these estimates are uniform in the ambient dimension.
Our main result improves the Boolean bound to
\begin{equation*}
\mathsf B_d^{\mathrm{Bool}}\le C d^{5/2}(1+\log d)^3,
\end{equation*}
and extends the estimates to finite local orthonormal systems.

\introrev{The Boolean arguments in
\cite{Ivanisvili2026,SloteTsengVolberg} build on the weighted square-function
method introduced in \cite{PellegrinoTeixeiraBH} for complex homogeneous
polynomials. In the discrete setting, coordinate partitions produce
sections whose coefficient norms must be estimated simultaneously.
We extend the simultaneous section estimate of Slote, Tseng and Volberg
\cite{SloteTsengVolberg} to finite local orthonormal systems.
}

The functions $1,x$ on $\{-1,1\}$ and the characters
$1,z,\ldots,z^{q-1}$ on the $q$th roots of unity are orthonormal
for normalized counting measure. We allow any local orthonormal
basis with first element $1$.

Finite cyclic groups place the Boolean cube in a larger family. For
$q\ge2$, let
\begin{equation*}
C_q:=\{e^{2\pi i k/q}:0\le k<q\}.
\end{equation*}
Every function $f:C_q^N\to\mathbb C$ has the unique expansion
\begin{equation*}
f(x)=
\sum_{\alpha\in\{0,\ldots,q-1\}^N}
\widehat f(\alpha)x^\alpha,
\qquad
\widehat f(\alpha)
=
q^{-N}\sum_{x\in C_q^N}f(x)\overline{x^\alpha}.
\end{equation*}
The chosen representatives define the canonical total degree $|\alpha|$.
Slote, Volberg and Zhang \cite{SloteVolbergZhang} established a
dimension-free Bohnenblust--Hille inequality for bounded total degree,
and Becker, Klein, Slote, Volberg and Zhang
\cite{BeckerKleinSloteVolbergZhang} obtained dimension-free
discretization estimates. In particular,
\cite[Corollary~7]{BeckerKleinSloteVolbergZhang} gives
\begin{equation*}
\left(
\sum_{|\alpha|\le d}
|\widehat f(\alpha)|^{p_d}
\right)^{1/p_d}
\le
(C\log q)^{2d}\|f\|_\infty
\end{equation*}
for bounded total degree.

\introrev{We use the interaction order}
\begin{equation*}
s(\alpha)
:=
\#\{j:\alpha_j\ne0\}.
\end{equation*}
It coincides with degree on the Boolean cube but, for $q>2$, ignores the
size of the nonzero local frequency and records only the number of
coordinates on which a character depends. Let
$\mathrm{BH}^{\mathrm{int}}_{d,q}$ denote the optimal dimension-free
constant, for functions of interaction order at most $d$, in
\begin{equation}
\left(
\sum_{s(\alpha)\le d}
|\widehat f(\alpha)|^{p_d}
\right)^{1/p_d}
\le
\mathrm{BH}^{\mathrm{int}}_{d,q}\|f\|_\infty.
\label{eq:intro-cyclic-interaction}
\end{equation}

A support-sensitive viewpoint and related local invariants on Hamming
schemes were developed by Defant, Galicer, Mansilla, Masty{\l}o and Muro
\cite{DefantEtAlHamming}. Subexponential growth for the larger
interaction-order class was obtained in \cite{PellegrinoRaposoCyclic}.
More recently, Slote and Volberg \cite[Theorem~1.1]{SloteVolbergPolynomial} combined the
weighted method of \cite{PellegrinoTeixeiraBH} with the Potts
hypercontractive estimate in \cite{PellegrinoRaposoCyclic} to obtain
polynomial growth. With
\begin{equation*}
\gamma_2=\frac12,
\qquad
\gamma_q=
\frac{q\log(q-1)}{4(q-2)}
\quad(q\ge3),
\end{equation*}
their estimate reads
\begin{equation}
\mathrm{BH}^{\mathrm{int}}_{d,q}
\le
C(q)d^{\,4\gamma_q+5}.
\label{eq:intro-SV-cyclic}
\end{equation}

On a fixed $q$-point space, the Bohnenblust--Hille constants
associated with two local orthonormal bases differ by at most a factor
$\sqrt{q-1}$ (Proposition~\ref{prop:change-local-basis}).

To obtain a Hardy--Littlewood inequality, we insert a complex variable
$r_j$ for each active coordinate $j$ of a Fourier term and take the
supremum over the unit ball of $\ell_p^N$. For an expansion of
interaction order exactly $d$, the resulting expression is
$d$-homogeneous in the variables $r_j$.

Hardy and Littlewood \cite{HardyLittlewood} proved a coefficient
inequality for bilinear forms, and Praciano-Pereira
\cite{PracianoPereira} gave its multilinear extension. We use the sharp
coefficient exponents of the complex polynomial Hardy--Littlewood
inequality \cite{AlbuquerqueBayartPellegrinoSeoane,DimantSevillaPeris}.
For $d\ge1$ and $d<p\le\infty$, set
\begin{equation}
\label{eq:rho}
\rho_{d,p}:=
\begin{cases}
\dfrac{p}{p-d},
& d<p\le2d,\\[1.2ex]
\dfrac{2dp}{dp+p-2d},
& 2d\le p<\infty,\\[1.2ex]
\dfrac{2d}{d+1},
& p=\infty.
\end{cases}
\end{equation}
The two expressions agree at $p=2d$, where $\rho_{d,2d}=2$;
at $p=\infty$, $\rho_{d,\infty}=2d/(d+1)=p_d$.
Theorem~C proves that these exponents remain optimal for the finite
linearly independent local systems considered here.

For $A\subseteq[N]:=\{1,\ldots,N\}$, put
$w_A(x):=\prod_{j\in A}x_j$. For a complex-valued function on the
Boolean cube \rev{with a homogeneous Walsh expansion of degree $d$}, write 
\begin{equation*}
f(x)=\sum_{|A|=d}a_Aw_A(x),\qquad  P_f(z)=\sum_{|A|=d}a_A\prod_{j\in
A}z_j\quad(z\in\mathbb{C}^N). 
\end{equation*}
For $d<p\le\infty$, introduce a complex radial vector $\mathbf r=(r_1,\ldots,r_N)$.
Define the norm 
\begin{equation*}
\|f\|_{p,\mathrm{Bool}}  :=\sup_{\|\mathbf r\|_{\ell_p^N}\le1}\max_{x\in\{-1,1\}^N} 
\left|\sum_{|A|=d}a_A\prod_{j\in A}r_jx_j\right| 
=\sup_{\|z\|_{\ell_p^N}\le1}|P_f(z)|. 
\end{equation*}
Indeed, multiplication by $x\in\{-1,1\}^N$ preserves the complex $\ell_p$
ball. The classical polynomial inequality therefore gives a dimension-free
coefficient estimate \rev{for these homogeneous Boolean polynomials}.

For the cyclic expansion $f(x)=\sum_{s(\alpha)=d}a_\alpha x^\alpha$, the
same prescription gives 
\begin{equation*}
\widetilde f(\mathbf r,x)  =\sum_{s(\alpha)=d}a_\alpha 
\left(\prod_{j:\alpha_j\ne0}r_j\right)x^\alpha,  \qquad (\mathbf r,x)\in\mathbb{C}%
^N\times C_q^N, 
\end{equation*}
and 
\begin{equation*}
\|f\|_{p,C_q}  =\sup_{\mathbf r\in\mathbb{C}^N,\,\|\mathbf r\|_p\le1}\max_{x\in
C_q^N}|\widetilde f(\mathbf r,x)|. 
\end{equation*}
The radial factor is $r_j$, irrespective of the value of the nonzero
frequency $\alpha_j$. Consequently the radial degree is exactly the
interaction order $d$. The coefficient inequality we study is 
\begin{equation*}
\left(\sum_{s(\alpha)=d}|a_\alpha|^{\rho_{d,p}}\right)^{1/\rho_{d,p}}  \le 
\mathsf{H}_{d,p}(C_q)\|f\|_{p,C_q}. 
\end{equation*}
We call $\mathsf{H}_{d,p}(C_q)$ the Hardy--Littlewood constant for these
Fourier expansions: it is the least constant for which the displayed
inequality holds for every dimension $N$ and \rev{every expansion in which
each term involves exactly $d$ coordinates}. Theorem~C establishes this inequality and proves that 
$\rho_{d,p}$ cannot be decreased in either the Boolean or the cyclic
setting. Theorem~B supplies quantitative bounds for $p\ge2d$ from the
discrete Bohnenblust--Hille constants.

At $p=\infty$, complex radial factors are still allowed, so the radial norm
can exceed the supremum on the finite group. The optimality in Theorem~C
concerns the coefficient exponent $\rho_{d,p}$.

To include the Boolean and cyclic cases in the same setting, we use product
bases on finite spaces. Theorem~A requires a complete orthonormal basis for
the uniform measure, while Theorem~C allows any linearly independent local
family containing the constants.

Fix an integer $q\ge2$. Let $\Omega$ be a set of $q$ points and let $\mu_q$
be the probability measure assigning mass $1/q$ to each point. Choose
functions 
\begin{equation*}
\varphi_u:\Omega\longrightarrow\mathbb{C}\qquad(0\le u\le q-1) 
\end{equation*}
forming an orthonormal basis $\Phi=(\varphi_0,\ldots,\varphi_{q-1})$ of $%
L_2(\Omega,\mu_q)$, with $\varphi_0\equiv1$. Thus 
\begin{equation*}
\frac1q\sum_{x\in\Omega}\varphi_u(x)\overline{\varphi_v(x)}  =%
\begin{cases}
1, & u=v, \\ 
0, & u\ne v.%
\end{cases}
\end{equation*}
The integer $q$ is the size of the local space. For each $N\ge1$, we
consider the product space $\Omega^N$, endowed with the product probability
measure $\mu_q^N$; $N$ is the number of coordinates.

For $\alpha\in\{0,\ldots,q-1\}^N$, define the product function 
\begin{equation*}
\Phi_\alpha:\Omega^N\longrightarrow\mathbb{C},\qquad 
\Phi_\alpha(x)=\prod_{j=1}^N\varphi_{\alpha_j}(x_j). 
\end{equation*}
These functions form an orthonormal basis of $L_2(\Omega^N,\mu_q^N)$.
Consequently every $f:\Omega^N\to\mathbb{C}$ has the unique expansion 
\begin{equation*}
f=\sum_{\alpha\in\{0,\ldots,q-1\}^N}\widehat f_\Phi(\alpha)\Phi_\alpha, 
\qquad  \widehat f_\Phi(\alpha)=\frac1{q^N}\sum_{x\in\Omega^N}  f(x)%
\overline{\Phi_\alpha(x)}. 
\end{equation*}
We use the norm $\|f\|_\infty=\max_{x\in\Omega^N}|f(x)|$.

Since $\varphi_0=1$, the active coordinate set of $\Phi_\alpha$ and its
cardinality are 
\begin{equation*}
\func{supp}(\alpha)=\{j\in\{1,\ldots,N\}:\alpha_j\ne0\},\qquad s(\alpha)=|%
\func{supp}(\alpha)|. 
\end{equation*}
For an integer $d\ge1$, we say that $f$ has interaction order at most $d$ if 
$\widehat f_\Phi(\alpha)=0$ whenever $s(\alpha)>d$.
\rev{We call the expansion homogeneous of order $d$ if its coefficients
vanish whenever $s(\alpha)\ne d$, so that each nonzero term involves
exactly $d$ coordinates. Here homogeneity refers to the number of active
coordinates, rather than to the sum of the local exponents.}

Two choices recover the settings discussed above:

\begin{itemize}
\item For $\Omega=\{-1,1\}$, $\varphi_0(x)=1$ and $\varphi_1(x)=x$, the
functions $\Phi_\alpha$ are the Walsh monomials. Interaction order is then
the usual degree of a Boolean polynomial.

\item For $\Omega=C_q=\{z\in\mathbb{C}:z^q=1\}$ and $\varphi_u(z)=z^u$, the
functions $\Phi_\alpha$ are the cyclic characters $z^\alpha$. Interaction
order counts the nonzero entries of $\alpha$, independently of their values.
This is the convention of \cite%
{PellegrinoRaposoCyclic,SloteVolbergPolynomial}.
\end{itemize}

The radial construction introduced for the Boolean and cyclic cases takes
the following form in this setting. For $d<p\le\infty$ and a
\rev{homogeneous expansion of order $d$},
$f=\sum_{s(\alpha)=d}a_\alpha\Phi_\alpha$, define 
\begin{equation*}
\widetilde f(\mathbf r,x)=\sum_{s(\alpha)=d}a_\alpha  \prod_{j\in\func{supp}%
\alpha}r_j\varphi_{\alpha_j}(x_j),  \qquad (\mathbf r,x)\in\mathbb{C}%
^N\times\Omega^N, 
\end{equation*}
and 
\begin{equation*}
\|f\|_{p,\Phi}=\sup_{\|\mathbf r\|_{\ell_p^N}\le1}\sup_{x\in\Omega^N}|\widetilde
f(\mathbf r,x)|. 
\end{equation*}
This definition recovers $\|f\|_{p,C_q}$ for the cyclic character
basis. With $\rho_{d,p}$ from \eqref{eq:rho}, the Hardy--Littlewood constant $%
\mathsf{H}_{d,p}(\Phi)$ is the least constant, uniform in $N$, for which 
\begin{equation*}
\left(\sum_{s(\alpha)=d}|a_\alpha|^{\rho_{d,p}}\right)^{1/\rho_{d,p}}  \le%
\mathsf{H}_{d,p}(\Phi)\|f\|_{p,\Phi}. 
\end{equation*}
Thus $\mathsf{H}_{d,p}(\Phi)$ is the supremum of the coefficient-to-norm
quotient over all dimensions and all nonzero \rev{homogeneous expansions
of order $d$}.
For the Boolean and cyclic bases, it is the Hardy--Littlewood constant for
the corresponding Fourier expansions.

We begin with the dimension-free Bohnenblust--Hille bound. Recall that $%
p_d=2d/(d+1)$. Let $\mathsf{B}_d^{\le}(\Omega,\Phi)$ be the least constant
for which 
\begin{equation*}
\left(\sum_{s(\alpha)\le d}|\widehat f_\Phi(\alpha)|^{p_d}\right)^{1/p_d} 
\le\mathsf{B}_d^{\le}(\Omega,\Phi)\|f\|_\infty 
\end{equation*}
holds for every $N\ge1$ and every $f:\Omega^N\to\mathbb{C}$ of interaction
order at most $d$. The constant $\mathsf{B}_d^{=}(\Omega,\Phi)$ is defined
by the same inequality restricted to \rev{homogeneous expansions of order
$d$}. Both
constants are uniform in $N$. For the cyclic character basis, we retain the
notation of \cite{PellegrinoRaposoCyclic}: 
\begin{equation*}
\mathrm{BH}^{\mathrm{int}}_{d,q}=\mathsf{B}_d^{\le}(C_q,\Phi). 
\end{equation*}
Set 
\begin{equation}  \label{eq:gamma-alpha-intro}
\gamma_2=\frac12,\qquad \gamma_q=\frac{q\log(q-1)}{4(q-2)}\quad(q\ge3),
\qquad \beta_q=\frac32+2\gamma_q.
\end{equation}

\begin{theoremA}
For every fixed $q\ge2$ there are $C(q),K(q)>0$ such that 
\begin{equation*}
\mathsf{B}_d^{\le}(\Omega,\Phi)  \le C(q)d^{\beta_q}(1+\log
d)^{K(q)}\qquad(d\ge1), 
\end{equation*}
uniformly over the orthonormal basis $\Phi$.
For the Boolean basis, one has the explicit bound
\begin{equation}
\label{eq:boolean-main-consequence}
\mathsf B_d^{\mathrm{Bool}}\le C d^{5/2}(1+\log d)^3.
\end{equation}
\end{theoremA}

Consequently, for every $\varepsilon>0$,
\[
\mathsf B_d^{\mathrm{Bool}}
=
O_\varepsilon(d^{5/2+\varepsilon}).
\]
\rev{The corresponding estimate in \cite{SloteTsengVolberg} is recorded
in \eqref{eq:intro-STV}.}

For the cyclic character basis,
\begin{equation}
\mathrm{BH}^{\mathrm{int}}_{d,q}
\le
C(q)d^{\,3/2+2\gamma_q}(1+\log d)^{K(q)}.
\label{eq:cyclic-main-consequence}
\end{equation}
\rev{For example, when $q=3$, the polynomial power in
\eqref{eq:cyclic-main-consequence} is $2.5397\ldots$.
The earlier estimate \eqref{eq:intro-SV-cyclic}, from
\cite{SloteVolbergPolynomial}, has power $5+4\gamma_q$.} The estimate is stated for interaction
order, but it also applies to the canonical total-degree class because
$s(\alpha)\le|\alpha|$. The same bound holds for
$\mathsf B_d^{=}(\Omega,\Phi)$ by restriction.
We now estimate the Hardy--Littlewood constant $\mathsf{H}_{d,p}(\Phi)$,
uniformly in the number of coordinates. Set 
\begin{equation}  \label{eq:Lambda-intro}
\Lambda_{d,q}=\pi\sqrt{q-1}\sqrt d\,(2d-1)\rev{e^{2\gamma_q(d-1)}}.
\end{equation}

\begin{theoremB}
For $q\ge2$, $d\ge2$ and $2d\le p\le\infty$, with $\tau=2d/p$
(and $\tau=0$ when $p=\infty$), 
\begin{equation*}
\mathsf{H}_{d,p}(\Phi)  \le\mathsf{B}_d^{=}(\Omega,\Phi)^{1-\tau}%
\Lambda_{d,q}^{\tau}. 
\end{equation*}
\end{theoremB}

\Needspace{9\baselineskip}

\begin{corollary}
\label{cor:HL-regimes} Fix $q\ge2$. Uniformly over $\Phi$, the following
hold as $d\to\infty$:

\begin{enumerate}
\item if $p(d)/d\to\infty$, then $\mathsf{H}_{d,p(d)}(\Phi)=\exp(o(d))$;

\item if $p(d)\ge cd^2/\log d$ for a fixed $c>0$, then 
\begin{equation*}
\mathsf{H}_{d,p(d)}(\Phi)\le C(q,c)  d^{\beta_q+\rev{4\gamma_q/c}}(1+\log
d)^{K(q)}; 
\end{equation*}

\item if $p(d)/d^2\to\infty$, then $\mathsf{H}_{d,p(d)}(\Phi)\le(1+o(1))%
\mathsf{B}_d^{=}(\Omega,\Phi)$.
\end{enumerate}
\end{corollary}

For $p=\infty$, the radial norm satisfies $\|f\|_\infty\le\|f\|_{\infty,\Phi}
$, and the inequality can be strict. On the Boolean cube, for example, 
\begin{equation*}
f(x)=x_1x_2+x_1x_3+i x_2x_3,  \qquad \|f\|_\infty=\sqrt5,\qquad
\|f\|_{\infty,\Phi}=3. 
\end{equation*}
Indeed, the triangle inequality bounds the radial norm by $3$, attained with 
$x=(1,1,1)$ and $\mathbf r=(e^{i\pi/4},e^{-i\pi/4},e^{-i\pi/4})$. The comparison in Corollary~\ref{cor:HL-regimes}
therefore does not identify the radial norm with the supremum norm on
the finite product space.

In Theorem~C the local functions need only be linearly independent.

Let $\nu\ge1$ and let
\[
\Phi=(\phi_0,\ldots,\phi_\nu)
\]
be a linearly independent family of functions on a finite set $\Omega$,
with $\phi_0\equiv1$. Define the radial norm and
$\mathsf{H}_{d,p}(\Phi)$ by the same formulas as above, with indices in
$\{0,\ldots,\nu\}^N$. Let $E_\Phi$ be the normed space on
$\mathbb C^\nu$ whose unit ball is
\[
B_{E_\Phi}
=
\operatorname{absconv}
\{
(\phi_1(x),\ldots,\phi_\nu(x)):
x\in\Omega
\},
\]
and set
\[
\kappa_{p,\Phi}
=
\|\operatorname{id}:\ell_p^\nu\to E_\Phi\|.
\]
Here $\operatorname{absconv}$ denotes the complex absolute convex hull.
Linear independence implies that the evaluation vectors span
$\mathbb C^\nu$, so the displayed set is indeed the unit ball of a norm.

Write $\mathsf{H}^{\mathrm{pol}}_{d,p}$ for the optimal constant in the
classical complex $d$-homogeneous polynomial Hardy--Littlewood
inequality on $\ell_p$.

\begin{theoremC}
Let $d\ge1$ and $d<p\le\infty$. For every finite linearly independent local
family $\Phi=(\phi_0,\ldots,\phi_\nu)$ with \rev{$\nu\ge1$ and} $\phi_0\equiv1$,
\begin{equation*}
\mathsf{H}_{d,p}(\Phi)
\le
\kappa_{p,\Phi}^{d}
\mathsf{H}^{\mathrm{pol}}_{d,p}.
\end{equation*}
Moreover, $\rho_{d,p}$ is optimal among dimension-free coefficient
exponents for every such local family.
\end{theoremC}

The constant $\kappa_{p,\Phi}$ depends only on the local family.
For optimality of the exponent, it suffices to consider square-free
polynomials involving one fixed nonconstant local function.

Theorems~A, B and C also apply to the six-point system in
Remark~\ref{rem:Tao-system}, whose orthonormal basis is not a cyclic
character table.

The proofs use the following estimates.

\begin{itemize}
\item \textbf{Potts hypercontractivity and independence of the local basis.}
\lastrev{On the Boolean cube we use the Bonami--Beckner hypercontractive inequality} \cite%
{Beckner,Bonami}. \attribrev{For a uniform $q$-point space, we use the Potts
\lastrev{hypercontractive estimate developed for cyclic coefficient inequalities in}
\cite{PellegrinoRaposoCyclic}, based on the logarithmic Sobolev method of
\cite{Gross} and the complete-graph inequality of \cite{DiaconisSaloffCoste}.
It gives the radius $(p-1)^{\gamma_q}$.} Since the Potts operator acts as the identity on constants and as a scalar
\lastrev{on their orthogonal complement, this estimate is independent of the local
orthonormal basis.}

\introrev{\item \textbf{Weighted coefficient estimates (Theorem~A).}
The weighted square-sum method of \cite{PellegrinoTeixeiraBH} is used in
the Boolean setting in \cite{Ivanisvili2026,SloteTsengVolberg} and in the
cyclic setting in \cite{SloteVolbergPolynomial}. We extend
\cite[Lemma~5.2]{SloteTsengVolberg} to local orthonormal systems.
For a fixed number $L$ of coordinate blocks, the admissible weight powers
satisfy $B>b_{q,L}$, with $b_{q,L}\downarrow2\gamma_q$ as $L\to\infty$.
We take $L$ of order $\log M$, where $M$ bounds the interaction order,
and use
\[
W_q(r)=r^{2\gamma_q}(1+\log r)^{K_q}.
\]
The logarithmic factor absorbs the remaining error in the contraction
estimate. The three-block argument gives the bound
$D_{q,3}M^{3/2}\|f\|_\infty$ at level $3$; the induction in
Proposition~\ref{prop:fixed-local-levels} retains the same power
at every fixed level. This gives Theorem~A; see Section~\ref{sec:proof-A}.}

\item \textbf{Rescaling and entropy cancellation (Theorem~B).} The
Hardy--Littlewood transfer uses the coefficient-dependent rescaling and
entropy cancellation of \cite[Section~3.2]{PellegrinoTeixeiraHL}. We
apply these steps to \rev{homogeneous expansions}, using a local square-function
estimate obtained from Potts hypercontractivity and the coefficient bounds
of \cite{LubinskyZiegler}.

\item \textbf{Convex geometry and optimal exponents (Theorem~C).} Theorem~C
uses the absolute convex hull of the local evaluation vectors to reduce the
radial norm to a polynomial norm on a sum of finite-dimensional spaces.
Classical polynomial Hardy--Littlewood estimates and the
Kahane--Salem--Zygmund inequality \cite{DefantMastylo,DimantSevillaPeris}
then give the coefficient exponents and their optimality.
\end{itemize}

\introrev{
All logarithms are natural. We use $\mathbb N=\{1,2,\ldots\}$ and
$1/p+1/p'=1$ for $1\le p\le\infty$, with $1/\infty=0$.
The local product spaces carry normalized counting measure.
}

\begin{table}[htbp]
\caption{Principal notation.}\centering
\begin{tabular}{@{}p{3.1cm}p{11.1cm}}
\toprule Symbol & Meaning \\ 
\midrule $d$, $N$, $q$ & Interaction bound, number of coordinates, and
number of local points. \\ 
$p_r=2r/(r+1)$ & Bohnenblust--Hille coefficient exponent. \\ 
$\Omega$, $\mu_q$ & Local $q$-point space and its uniform probability
measure. \\ 
$\Phi$, $\Phi_\alpha$ & Local family and its product functions; $\Phi$ is
orthonormal in Theorems A and B. \\ 
$\func{supp}(\alpha)$, $s(\alpha)$ & Active coordinate set and its
cardinality. \\ 
$\mathbf r=(r_j)$ & Complex radial vector in the norm $\|f\|_{p,\Phi}$. \\ 
$\rho_{d,p}$ & Hardy--Littlewood coefficient exponent in \eqref{eq:rho}. \\ 
$\mathsf{B}_d^{\le}$, $\mathsf{B}_d^{=}$ & \rev{Optimal constants for at most
$d$ or exactly $d$ coordinates per term.} \\ 
$\mathrm{BH}^{\mathrm{int}}_{d,q}$ & The cyclic specialization of $\mathsf{B}%
_d^{\le}$. \\ 
$\mathsf{H}_{d,p}(\Phi)$ & Optimal dimension-free Hardy--Littlewood constant
for the expansion in $\Phi$. \\ 
$\gamma_q$, $\beta_q$ & Potts exponent and polynomial exponent $3/2+2\gamma_q
$. \\ 
\bottomrule
\end{tabular}%
\end{table}

\clearpage
\section*{\introrev{Proof dependencies}}
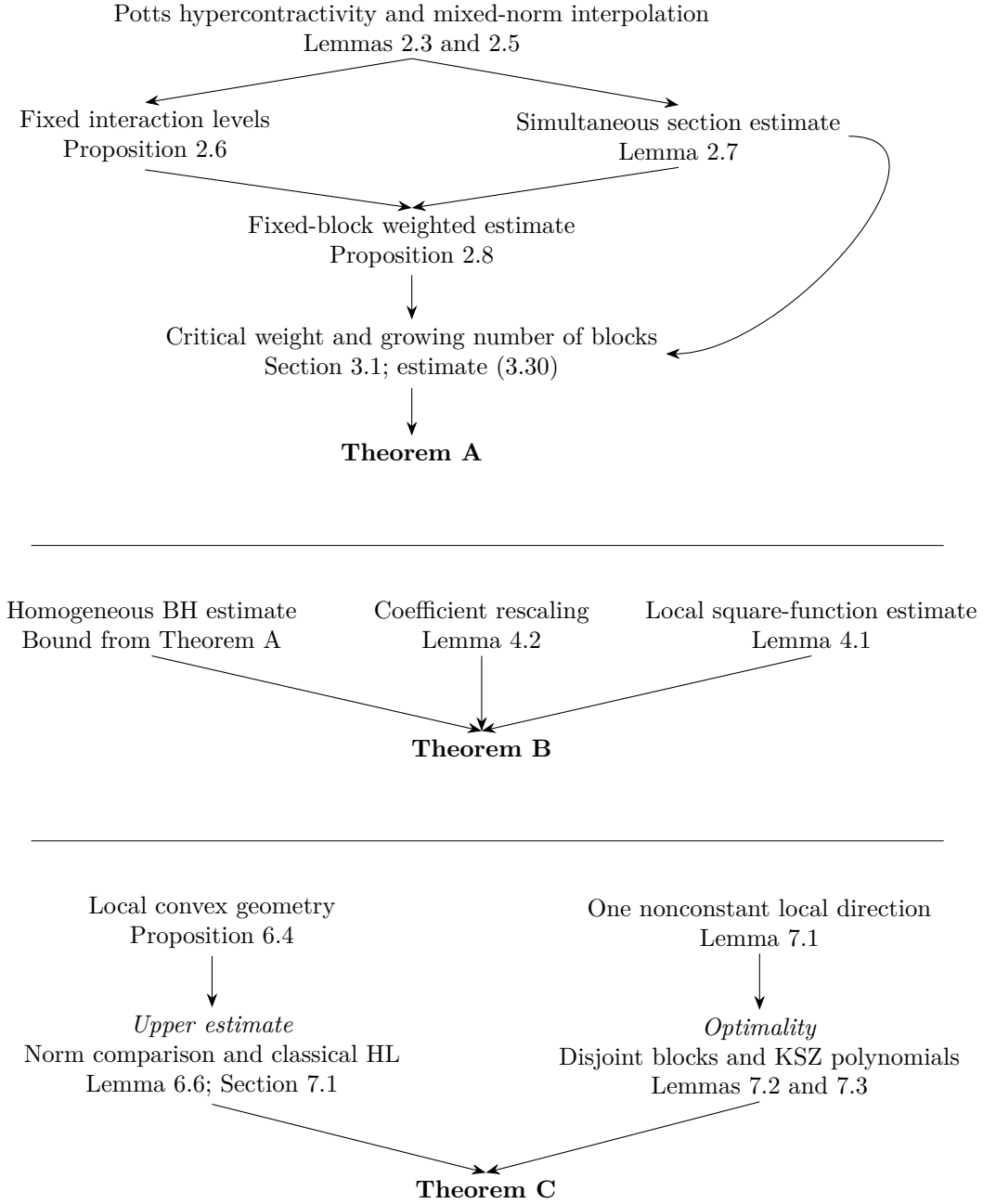
\begin{figure}[H]
\centering
\begin{tikzpicture}[
  every node/.style={align=center,font=\small,text=black,inner sep=3pt},
  >={Stealth[length=2mm]},
  dep/.style={->,draw=black,line width=.4pt}]
\node (a0) at (0,0) {Potts hypercontractivity and mixed-norm interpolation\\
Lemmas~\ref{lem:potts-hc-general} and~\ref{lem:local-complement-interpolation}};
\node (a1) at (-3.8,-1.55) {Fixed interaction levels\\Proposition~\ref{prop:fixed-local-levels}};
\node (a2) at (3.8,-1.55) {Simultaneous section estimate\\Lemma~\ref{lem:simultaneous-local-sections}};
\node (a3) at (0,-3.05) {Fixed-block weighted estimate\\Proposition~\ref{prop:fixed-block-local}};
\node (a4) at (0,-4.65) {Critical weight and growing number of blocks\\
Section~\ref{sec:critical-weight}; estimate~\eqref{eq:critical-local-Gamma-final}};
\node[font=\small\bfseries] (a5) at (0,-6.05) {Theorem A};
\draw[dep] (a0.south) -- (a1.north);
\draw[dep] (a0.south) -- (a2.north);
\draw[dep] (a1.south) -- (a3.north);
\draw[dep] (a2.south) -- (a3.north);
\draw[dep] (a3.south) -- (a4.north);
\draw[dep] (a2.east) to[out=0,in=0,looseness=1.1] (a4.east);
\draw[dep] (a4.south) -- (a5.north);
\end{tikzpicture}

\vspace{.65cm}
\rule{13cm}{.3pt}
\vspace{.65cm}

\begin{tikzpicture}[
  every node/.style={align=center,font=\small,text=black,inner sep=3pt},
  >={Stealth[length=2mm]},
  dep/.style={->,draw=black,line width=.4pt}]
\node (b0) at (-4.7,0) {Homogeneous BH estimate\\Bound from Theorem A};
\node (b1) at (0,0) {Coefficient rescaling\\Lemma~\ref{lem:local-rescaling}};
\node (b2) at (4.7,0) {Local square-function estimate\\Lemma~\ref{lem:mixed-one-block}};
\node[font=\small\bfseries] (b3) at (0,-1.75) {Theorem B};
\draw[dep] (b0.south) -- (b3.north);
\draw[dep] (b1.south) -- (b3.north);
\draw[dep] (b2.south) -- (b3.north);
\end{tikzpicture}

\vspace{.65cm}
\rule{13cm}{.3pt}
\vspace{.65cm}

\begin{tikzpicture}[
  every node/.style={align=center,font=\small,text=black,inner sep=3pt},
  >={Stealth[length=2mm]},
  dep/.style={->,draw=black,line width=.4pt}]
\node (c0) at (-3.9,0) {Local convex geometry\\Proposition~\ref{prop:geometry}};
\node (c1) at (3.9,0) {One nonconstant local direction\\Lemma~\ref{lem:squarefree}};
\node (c2) at (-3.9,-1.9) {\textit{Upper estimate}\\
Norm comparison and classical HL\\Lemma~\ref{lem:comparison}; Section~\ref{sec:HL-estimate}};
\node (c3) at (3.9,-1.9) {\textit{Optimality}\\
Disjoint blocks and KSZ polynomials\\Lemmas~\ref{lem:disjoint} and~\ref{lem:KSZ}};
\node[font=\small\bfseries] (c4) at (0,-3.8) {Theorem C};
\draw[dep] (c0.south) -- (c2.north);
\draw[dep] (c1.south) -- (c3.north);
\draw[dep] (c2.south) -- (c4.north);
\draw[dep] (c3.south) -- (c4.north);
\end{tikzpicture}
\caption{\introrev{From top to bottom: the estimates leading to Theorems A, B and C.}}
\label{fig:proof-dependencies}
\end{figure}
\clearpage

\section{Auxiliary results for the Bohnenblust--Hille estimate}
\label{sec:BH-auxiliary}

Fix $q\ge2$ and a local orthonormal basis $\Phi$ as in
the introduction. For $r\ge1$ put 
\begin{equation*}
\mathcal{A}_r(f)=\left(\sum_{s(\alpha)=r}|\widehat
f_\Phi(\alpha)|^{p_r}\right)^{1/p_r}. 
\end{equation*}
Here $r$ is a positive integer. Constants denoted by $C(q)$ may change
from line to line and are uniform in the dimension and the local basis.
Empty coordinate blocks are allowed; $\Omega^{\varnothing}$ is a singleton
and coefficient sums over empty index sets are zero.

\introrev{
For $N\ge1$ and $L\ge2$, let
\[
\Pi_{N,L}:=\{\omega:\{1,\ldots,N\}\to\{1,\ldots,L\}\},
\qquad
\nu_{N,L}(\{\omega\}):=L^{-N}.
\]
The map $\omega$ determines the blocks
$I_j(\omega):=\{k:\omega(k)=j\}$, $1\le j\le L$.
Under $\nu_{N,L}$ the coordinate assignments are independent and uniform.
For $F:\Pi_{N,L}\to\mathbb C$, write
\[
\int F(\omega)\,d\omega
:=\int_{\Pi_{N,L}}F(\omega)\,d\nu_{N,L}(\omega)
=L^{-N}\sum_{\omega\in\Pi_{N,L}}F(\omega)
\]
with $d\omega=d\nu_{N,L}(\omega)$.
}

\Needspace{10\baselineskip}
{
\begin{proposition}[Change of the local orthonormal basis]
\label{prop:change-local-basis}
Let $\Phi$ and $\Psi$ be orthonormal bases of $L_2(\Omega,\mu_q)$,
both with first element $1$. Interaction order and the
radial extension are independent of this choice of basis. In particular,
for every homogeneous expansion $f$ of order $d\ge1$ and every
$d<p\le\infty$, $\|f\|_{p,\Phi}=\|f\|_{p,\Psi}$. For
$d\ge1$ and $\star\in\{\le,=\}$,
\begin{equation}
\label{eq:BH-basis-comparison}
\mathsf B_d^\star(\Omega,\Phi)
\le \sqrt{q-1}\,\mathsf B_d^\star(\Omega,\Psi).
\end{equation}
For $d<p\le\infty$,
\begin{equation}
\label{eq:HL-basis-comparison}
\mathsf H_{d,p}(\Phi)
\le (q-1)^{d|1/\rho_{d,p}-1/2|}\mathsf H_{d,p}(\Psi).
\end{equation}
Both inequalities also hold with $\Phi$ and $\Psi$ interchanged. When
$p\ge2d$, the factor in \eqref{eq:HL-basis-comparison} is
$(q-1)^{1/2-d/p}\le\sqrt{q-1}$; at $p=2d$, the two
Hardy--Littlewood constants coincide.
\end{proposition}

\begin{proof}
For each support $S\subseteq\{1,\ldots,N\}$, the product functions with
support exactly $S$ form an orthonormal basis of the tensor product of
the mean-zero local subspaces in $S$ and the constant subspaces outside
$S$. This subspace does not depend on the local basis. Thus its component
$f_S$ in the expansion of $f$ is basis independent, and the radial
extension is $\sum_S f_S(x)\prod_{j\in S}r_j$ in either basis.

On a support of size $s$, the coefficient vectors are related by a
unitary map of dimension $m=(q-1)^s$. The elementary comparison of
finite-dimensional norms gives
\[
\|Uc\|_t\le m^{|1/t-1/2|}\|c\|_t
\qquad(1\le t\le\infty)
\]
for every unitary $U$ on $\mathbb C^m$. For $t=p_d=2d/(d+1)$ and
$s\le d$, this factor is $(q-1)^{s/(2d)}\le\sqrt{q-1}$.
Taking the $p_d$-sum over supports and then the supremum in the
definition of the Bohnenblust--Hille constants proves
\eqref{eq:BH-basis-comparison}. The same argument with $s=d$ and
$t=\rho_{d,p}$, together with equality of the radial norms, proves
\eqref{eq:HL-basis-comparison}.
\end{proof}
}

\subsection{The first interaction level}

\leavevmode\attribrev{The phase-averaging argument used for linear Walsh polynomials in
\cite[Lemma~3.3]{SloteTsengVolberg} gives the following form for mean-zero
local functions.}

\begin{lemma}
\label{lem:phase-sum} Let $h_1,\ldots,h_N$ be mean-zero complex functions on
a finite probability space $\Omega$. Then 
\begin{equation*}
\sum_{j=1}^N\|h_j\|_\infty\le\pi\max_{x\in\Omega^N}\left|%
\sum_{j=1}^Nh_j(x_j)\right|. 
\end{equation*}
\end{lemma}

\begin{proof}
Choose $a_j$ with $h_j(a_j)=\|h_j\|_\infty e^{i\theta_j}$. For $\theta\in[%
0,2\pi]$ set $m_j(\theta)=\max_x\func{Re}(e^{-i\theta}h_j(x))$. Mean zero
gives $m_j(\theta)\ge0$, so 
\begin{equation*}
m_j(\theta)\ge\|h_j\|_\infty\max\{\cos(\theta_j-\theta),0\}. 
\end{equation*}
The integral of the positive part of the cosine on $[0,2\pi]$ is $2$. Thus
the average of $\sum_jm_j(\theta)$ is at least $\pi^{-1}\sum_j\|h_j\|_\infty$%
. For a phase attaining at least this average, choose each coordinate at a
corresponding maximum and take the real part of the sum.
\end{proof}

\subsection{The Potts semigroup and fixed-level extraction}

\label{sec:potts} For one coordinate let $E,T_\rho:L_2(\Omega,\mu_q)\to
L_2(\Omega,\mu_q)$ be given by 
\begin{equation*}
Eg:=\left(\int_{\Omega}g\,d\mu_q\right)\mathbf{1},  \qquad 
T_\rho:=E+\rho(I-E),  \qquad 0\le\rho\le1. 
\end{equation*}
Since $\varphi_0\equiv1$ and every $\varphi_u$, $u\ge1$, is orthogonal to
the constants, 
\begin{equation*}
T_\rho\varphi_0=\varphi_0,  \qquad  T_\rho\varphi_u=\rho\varphi_u\quad(1\le
u\le q-1). 
\end{equation*}
Tensorization gives 
\begin{equation}  \label{eq:Potts-diagonal-action}
T_\rho^{\otimes N}\Phi_\alpha =\rho^{s(\alpha)}\Phi_\alpha.
\end{equation}
We use 
\begin{equation}  \label{eq:gamma-q}
\gamma_2:=\frac12, \qquad \gamma_q:=\frac{q\log(q-1)}{4(q-2)}\quad(q\ge3).
\end{equation}

\attribrev{We use the Potts hypercontractive estimate from
\cite[Proposition~3.1 and Appendix~A]{PellegrinoRaposoCyclic}, which follows
from the complete-graph log--Sobolev inequality and Gross' logarithmic
Sobolev method; see \cite[Theorem~3.5 and Corollary~A.5]{DiaconisSaloffCoste}
and \cite{Gross}.} The operator
depends only on the averaging projection, so the estimate is independent of
the basis. Positivity gives $|T_\rho^{\otimes N}g|\le T_\rho^{\otimes N}|g|$%
, and hence the real inequality also holds for complex functions.

\begin{lemma}
\label{lem:potts-hc-general} Let $1<p\le2$. For every $N$ and every $%
g:\Omega^N\to\mathbb{C}$, 
\begin{equation}  \label{eq:Potts-hypercontractivity}
\|T_\rho^{\otimes N}g\|_2\le\|g\|_p \qquad\text{whenever}\qquad
0\le\rho\le(p-1)^{\gamma_q}.
\end{equation}
Consequently, for every integer $k\ge0$, if
$\widehat g_\Phi(\alpha)=0$ whenever $s(\alpha)>k$, then 
\begin{equation}  \label{eq:low-level-Potts}
\|g\|_2\le(p-1)^{-\gamma_q k}\|g\|_p.
\end{equation}
\end{lemma}

The consequence \eqref{eq:low-level-Potts} is immediate from %
\eqref{eq:Potts-diagonal-action} and Parseval: with $\rho_0=(p-1)^{\gamma_q}$%
, 
\begin{equation*}
\rho_0^k\|g\|_2  \le \|T_{\rho_0}^{\otimes N}g\|_2  \le \|g\|_p. 
\end{equation*}

\introrev{
Let $P_1:L_2(\Omega^N,\mu_q^N)\to
\operatorname{span}\{\Phi_\alpha:s(\alpha)=1\}$ be the orthogonal projection
\[
P_1g:=\sum_{s(\alpha)=1}\widehat g_\Phi(\alpha)\Phi_\alpha.
\]
}

\attribrev{For $q=2$, the next estimate is the first-level case of
\cite[Lemma~3.2]{SloteTsengVolberg}. We use the Potts estimate above to
obtain its local analogue.}

\begin{lemma}
\label{lem:Potts-L1-L2-quantitative} For every $M\ge1$ and every $g$ of
interaction order at most $M$, 
\begin{equation}  \label{cyc:projection-bound}
\|P_1g\|_2\le C(q)\sqrt M\,\|g\|_1.
\end{equation}
\end{lemma}

\begin{proof}
Parseval and \eqref{eq:Potts-hypercontractivity} imply $\|g\|_2\le(p-1)^{-%
\gamma_qM}\|g\|_p$ for $1<p<2$. With $\vartheta=(2-p)/p$, interpolation
gives $\|g\|_p\le\|g\|_1^{\vartheta}\|g\|_2^{1-\vartheta}$. For $g\ne0$,
\lastrev{these inequalities give
\[
\|g\|_2\le
\exp\left(\gamma_q M\frac{-p\log(p-1)}{2-p}\right)\|g\|_1.
\]
Letting $p\uparrow2$ and using
$\lim_{p\uparrow2}[-p\log(p-1)/(2-p)]=2$, we obtain}
\begin{equation}  \label{cyc:reverse-moment}
\|g\|_2\le e^{2\gamma_qM}\|g\|_1.
\end{equation}

\lastrev{Write $h=P_1g=\sum_{i=1}^N h_i(x_i)$, where
$\int_\Omega h_i\,d\mu_q=0$. Orthogonality and the uniform probability
measure on $\Omega$ give}
\begin{equation*}
\lastrev{\sum_{i=1}^N\|h_i\|_2^2=\|h\|_2^2},  \qquad \|h_i\|_\infty\le\sqrt q\,\|h_i\|_2.
\end{equation*}
For $h\ne0$, Hoeffding's exponential-moment
bound for independent centered bounded variables \cite[Section~2.6]%
{BoucheronLugosiMassart}, applied to its real and imaginary parts, gives 
\begin{equation*}
\lastrev{\mu_q^N\{x\in\Omega^N:|h(x)|>u\}\le4\exp\left(-\frac{u^2}{4q\|h\|_2^2}\right)}
\qquad(u\ge0). 
\end{equation*}
\lastrev{For $v\ge2$, let $n=\lceil v/2\rceil$. Tonelli's theorem and the
substitution $t=u^2/(4q\|h\|_2^2)$ give
\begin{align*}
\|h\|_{2n}^{2n}
&=2n\int_0^\infty u^{2n-1}
\mu_q^N\{x\in\Omega^N:|h(x)|>u\}\,du\\
&\le8n\int_0^\infty u^{2n-1}
\exp\left(-\frac{u^2}{4q\|h\|_2^2}\right)\,du\\
&=4n!\,(4q\|h\|_2^2)^n
\le4n^n(4q\|h\|_2^2)^n.
\end{align*}
Here $\int_0^\infty t^{n-1}e^{-t}\,dt=(n-1)!$, by integration by parts.
Since $\|h\|_v\le\|h\|_{2n}$ and $n\le v$, this gives}
\begin{equation}  \label{cyc:linear-moments}
\|h\|_v\le C\sqrt{qv}\,\|h\|_2\qquad(v\ge2).
\end{equation}
The assertions are immediate when $h=0$. Otherwise, choose $p=1+1/M$, so $%
p^{\prime }=M+1$. Interpolation between $L^1$ and $L^2$, followed by %
\eqref{cyc:reverse-moment}, gives 
\begin{equation*}
\|g\|_p\le\|g\|_1^{(M-1)/(M+1)}  \|g\|_2^{2/(M+1)}  \le
e^{4\gamma_q}\|g\|_1. 
\end{equation*}
\lastrev{Since $P_1$ is an orthogonal projection, H\"older's inequality gives}
\begin{equation*}
\lastrev{\|h\|_2^2=\left|\int_{\Omega^N}g\overline h\,d\mu_q^N\right|}
\le\|g\|_p\|h\|_{p^{\prime }}  \le C
e^{4\gamma_q}\sqrt{q(M+1)}\,\|g\|_1\|h\|_2. 
\end{equation*}
Dividing by $\|h\|_2$ proves \eqref{cyc:projection-bound}.
\end{proof}

\subsection{Multi-block interpolation and local noise}

\label{sec:multiblock-local} Let $L\ge2$ and let the coordinate set be
written as a disjoint union $I_1\dot\cup\cdots\dot\cup I_L$.
For $1\le j\le L$, write $x_j\in\Omega^{I_j}$ for the corresponding block
variable and $\alpha^{(j)}\in\{0,\ldots,q-1\}^{I_j}$ for a block multi-index.
A coefficient
of a function on $\Omega^{I_1}\times\cdots\times\Omega^{I_L}$ is indexed by
a tuple $(\alpha^{(1)},\ldots,\alpha^{(L)})$. For a nonzero multidegree $%
\mathbf{d}=(d_1,\ldots,d_L)\in\mathbb{N}_0^L$ put 
\begin{equation}  \label{eq:local-degrees-theta}
r:=|\mathbf{d}|=\sum_{j=1}^Ld_j,\qquad s_j:=r-d_j,\qquad \theta_j:=\frac{s_j%
}{(L-1)r}.
\end{equation}
When $s_j\ge1$ for every $j$, one has $\theta_j>0$ and $\introrev{\sum_{j=1}^L}\theta_j=1$.

For a function $Q$ written in these blocks, denote its coefficient array at
multidegree $\mathbf{d}$ by 
\begin{equation*}
a_\mathbf{d}(Q):=\bigl(\widehat Q_\Phi(\alpha^{(1)},\ldots,\alpha^{(L)})%
\bigr)_ {s(\alpha^{(j)})=d_j\ (1\le j\le L)}. 
\end{equation*}
Its total interaction order is $r$. The number $s_j$ is the interaction
order in all blocks except $I_j$.

For a scalar array $a=(a_{i_1,\ldots,i_L})$ indexed by finite sets $%
E_1\times\cdots\times E_L$ and $s_j\ge1$, define 
\begin{equation}  \label{eq:local-Xj}
X_j(a):= \left[ \sum_{(i_k)_{k\ne j}}
\left(\sum_{i_j}|a_{i_1,\ldots,i_L}|^2\right)^{p_{s_j}/2} \right]%
^{1/p_{s_j}}.
\end{equation}

\attribrev{The following is a Blei-type mixed-norm interpolation inequality;
see \cite{BayartPellegrinoSeoane}. Its two-block form appears in
\cite[Lemma~3.4]{PellegrinoTeixeiraBH}.}

\begin{lemma}
\label{lem:local-complement-interpolation} Assume that $s_j\ge1$ for every $j
$. With \eqref{eq:local-degrees-theta}, 
\begin{equation}  \label{eq:local-complement-interpolation}
\|a\|_{\ell^{p_r}} \le \prod_{j=1}^L X_j(a)^{\theta_j}.
\end{equation}
\end{lemma}

\begin{proof}
Fix an order of the indices. For each $j$, assign exponent $2$ to the $j$-th
index and exponent $p_{s_j}$ to every other index, and denote the resulting
iterated norm by $N_j(a)$. Since $p_{s_j}\le2$, repeated Minkowski gives $%
N_j(a)\le X_j(a)$. For a fixed index position $k$, 
\begin{align*}
\frac{\theta_k}{2}+\sum_{j\ne k}\frac{\theta_j}{p_{s_j}} &=\frac12+\sum_{j%
\ne k}\frac{\theta_j}{2s_j} =\frac12+\frac1{2r} =\frac1{p_r}.
\end{align*}
Repeated H\"older interpolation over the index positions therefore gives $%
\|a\|_{\ell^{p_r}}\le\introrev{\prod_{j=1}^L}N_j(a)^{\theta_j}$, and hence %
\eqref{eq:local-complement-interpolation}.
\end{proof}

For $s\ge2$ put 
\begin{equation}  \label{eq:local-sigma-s}
\sigma_s:=\left(\frac{s-1}{s+1}\right)^{\gamma_q}.
\end{equation}
By Lemma~\ref{lem:potts-hc-general}, if $J$ is a coordinate block and $%
G:\Omega^J\to\mathbb{C}$, then for every integer $t\ge0$, 
\begin{equation}  \label{eq:local-single-level-noise}
\sigma_s^t \left(\sum_{s(\xi)=t}|\widehat G_\Phi(\xi)|^2\right)^{1/2} \le
\|T_{\sigma_s}^{\otimes J}G\|_2\le\|G\|_{p_s}.
\end{equation}
\introrev{The function $G$ in \eqref{eq:local-single-level-noise} may contain all
interaction levels. The first inequality follows by restricting the
nonnegative Parseval sum to the indices with $s(\xi)=t$.}

\subsection{Fixed interaction levels}

\label{sec:fixed-local-levels}

\begin{proposition}
\label{prop:fixed-local-levels} For every fixed $q\ge2$ and $r\ge1$ there is 
$D_{q,r}<\infty$, independent of $M$, the ambient dimension and the
orthonormal basis $\Phi$, such that 
\begin{equation}  \label{eq:fixed-local-levels}
\mathcal{A}_r(f)\le D_{q,r}M^{3/2}\|f\|_\infty
\end{equation}
whenever $f$ has interaction order at most $M$ and $r\le M$.
\end{proposition}

\begin{proof}
For $r=1$, write $f=\sum_{k=0}^Mf_k$ according to interaction order.
Extend the definition algebraically by $T_t:=E+t(I-E)$ for
$t\in\mathbb R$. On the interval $[-1/(q-1),1]$, this
one-coordinate operator is Markov.
\attribrev{For the use of this interval in spectral extraction, see
\cite[Lemmas~3.3 and~3.5]{SloteVolbergPolynomial}.}
Thus, for fixed $x$, 
\begin{equation*}
P_x(t):=(T_t^{\otimes N}f)(x)=\sum_{k=0}^Mt^kf_k(x) 
\end{equation*}
is bounded by $\|f\|_\infty$ on a fixed interval depending only on $q$.
Bernstein's inequality at the interior point $0$
\cite[equation~(4.1)]{AdcockShadrin} gives $\|f_1\|_\infty\le C(q)M\|f\|_\infty$. Writing $%
f_1=\sum_{j=1}^N h_j(x_j)$ and using Cauchy--Schwarz, Parseval and Lemma~\ref%
{lem:phase-sum}, 
\begin{equation*}
\mathcal{A}_1(f)  \le \pi\sqrt{q-1}\,\|f_1\|_\infty  \le C(q)M\|f\|_\infty 
\le C(q)M^{3/2}\|f\|_\infty. 
\end{equation*}

For $r=2$, \attribrev{we adapt the proof of Slote, Tseng and Volberg
\cite[Proposition~4.3]{SloteTsengVolberg}.} Assign each coordinate independently to $I$ or $J$ with
probability $1/2$. For a fixed partition write 
\begin{equation*}
f(x,y)=\sum_{\alpha,\xi}b_{\alpha,\xi}\Phi_\alpha(x)\Phi_\xi(y),  \qquad 
B_\alpha(y):=\sum_\xi b_{\alpha,\xi}\Phi_\xi(y). 
\end{equation*}
Here $x\in\Omega^I$, $y\in\Omega^J$, and $B_\alpha:\Omega^J\to\mathbb{C}$.
\lastrev{Consider the array $a=(b_{\alpha,\xi})_{s(\alpha)=s(\xi)=1}$ and define its}
two mixed norms by 
\begin{equation*}
X:=\sum_{s(\alpha)=1}  \left(\sum_{s(\xi)=1}|b_{\alpha,\xi}|^2\right)^{1/2},
\qquad  Y:=\sum_{s(\xi)=1} 
\left(\sum_{s(\alpha)=1}|b_{\alpha,\xi}|^2\right)^{1/2}. 
\end{equation*}
Parseval and Lemma~\ref{lem:Potts-L1-L2-quantitative} give 
\begin{align*}
X&=\sum_{s(\alpha)=1}\|P_1B_\alpha\|_2 \le C(q)\sqrt
M\sum_{s(\alpha)=1}\|B_\alpha\|_1 \\
&=C(q)\sqrt M\int_{\Omega^J} \sum_{s(\alpha)=1}|B_\alpha(y)|\,d\mu_q^J(y) \\
&=C(q)\sqrt M\int_{\Omega^J}\mathcal{A}_1(f(\cdot,y))\,d\mu_q^J(y) \le
C^{\prime}(q)M^{3/2}\|f\|_\infty.
\end{align*}
The last step uses the level-one estimate proved above on each section.
Interchanging $I$ and $J$ gives the same bound for $Y$. The case $L=2$ of
Lemma~\ref{lem:local-complement-interpolation} yields $\|a\|_{p_2}\le
X^{1/2}Y^{1/2}\le C^{\prime}(q)M^{3/2}\|f\|_\infty$. Each coefficient of
\introrev{interaction order two has one active coordinate in each block with
probability $1/2$. Let $a_\omega$ be the coefficient array indexed by
$s(\alpha)=s(\xi)=1$ for $\omega\in\Pi_{N,2}$. With the uniform measure
defined above,} 
\begin{equation*}
\frac12\mathcal{A}_2(f)^{p_2}  =\int\|a_\omega\|_{p_2}^{p_2}\,d\omega 
\le(C^{\prime}(q)M^{3/2}\|f\|_\infty)^{p_2}. 
\end{equation*}
Consequently, 
\begin{equation}  \label{eq:fixed-level-two-local}
\mathcal{A}_2(f)\le D_{q,2}M^{3/2}\|f\|_\infty.
\end{equation}

For $r=3$, partition the coordinates uniformly into three blocks
\[
I_1\dot\cup I_2\dot\cup I_3=\{1,\ldots,N\}.
\]
Write $Q$ for the same function $f$ viewed, according to this partition,
as a function on
\[
\Omega^{I_1}\times\Omega^{I_2}\times\Omega^{I_3},
\]
and write
\[
\widehat Q_\Phi
\bigl(\alpha^{(1)},\alpha^{(2)},\alpha^{(3)}\bigr)
\]
for the corresponding regrouped coefficient array. \introrev{We consider the coefficients of
multidegree $(1,1,1)$.} Thus
\[
s_j=2,
\qquad
\theta_j=\frac13,
\qquad
1\le j\le3.
\]
\lastrev{For each mixed norm, we use the complete coefficient functions of
sections of $Q$. In particular,}
\[
\|Q\|_\infty=\|f\|_\infty.
\]

Fix $j$, and write $I_{\widehat j}:=\bigcup_{k\ne j}I_k$. For a
complementary index
\[
\beta=(\alpha^{(k)})_{k\ne j},
\qquad
s(\alpha^{(k)})=1,
\]
define the coefficient function in the $j$-th block by
\[
G_\beta(x_j)
:=
\sum_{\alpha^{(j)}\in\{0,\ldots,q-1\}^{I_j}}
\widehat Q_\Phi
\bigl(\alpha^{(1)},\alpha^{(2)},\alpha^{(3)}\bigr)
\Phi_{\alpha^{(j)}}(x_j),
\]
where the complementary indices are fixed equal to $\beta$. \lastrev{The sum
defining $G_\beta$ includes every interaction level in the $j$-th block.}

Applying \eqref{eq:local-single-level-noise} to $G_\beta$ with
$s=2$ and $t=1$ gives
\[
\sigma_2
\left(
\sum_{s(\alpha^{(j)})=1}
\left|
\widehat Q_\Phi
(\alpha^{(1)},\alpha^{(2)},\alpha^{(3)})
\right|^2
\right)^{1/2}
\le
\|G_\beta\|_{p_2}.
\]
Taking the $\ell^{p_2}$ norm over the complementary indices $\beta$
and using Fubini yields
\begin{align*}
\sigma_2 X_j
&\le
\left(
\int_{\Omega^{I_j}}
\sum_{\substack{s(\alpha^{(k)})=1\\ k\ne j}}
\left|
G_\beta(x_j)
\right|^{p_2}
\,d\mu_q^{I_j}(x_j)
\right)^{1/p_2}.
\end{align*}
For each fixed $x_j$, the expression inside the integral is bounded by
the \rev{$p_2$-th power of the} level-two coefficient norm of the complementary section
\[
Q(x_j,\cdot):
\Omega^{I_{\widehat j}}\longrightarrow\mathbb C.
\]
This section still has interaction order at most $M$, and
\[
\|Q(x_j,\cdot)\|_\infty\le\|Q\|_\infty=\|f\|_\infty.
\]
Therefore \eqref{eq:fixed-level-two-local} gives
\[
\left(
\sum_{\substack{s(\alpha^{(k)})=1\\ k\ne j}}
|G_\beta(x_j)|^{p_2}
\right)^{1/p_2}
\le
D_{q,2}M^{3/2}\|f\|_\infty
\]
for every $x_j$. Consequently,
\begin{equation}
X_j
\le
\sigma_2^{-1}D_{q,2}M^{3/2}\|f\|_\infty,
\qquad
1\le j\le3.
\label{eq:fixed-level-three-Xj}
\end{equation}
Lemma~\ref{lem:local-complement-interpolation} gives the same bound for the $%
\ell^{p_3}$ norm of \introrev{the coefficient array of multidegree $(1,1,1)$.} A support of size three
falls with one active coordinate in each block with probability $3!/3^3=2/9$%
. Averaging the $p_3$-th power therefore gives 
\begin{equation}  \label{eq:fixed-level-three-local}
\mathcal{A}_3(f)\le D_{q,3}M^{3/2}\|f\|_\infty.
\end{equation}

Let now $r\ge4$ and assume \eqref{eq:fixed-local-levels} at level $r-2$. Use
\introrev{a random two-block partition and consider the coefficients of
multidegree $(2,r-2)$.} With $X$
denoting the mixed norm with outer exponent $p_2$ and $Y$ the mixed norm
with outer exponent $p_{r-2}$, \eqref{eq:local-single-level-noise} gives 
\begin{equation*}
X\le \sigma_2^{-(r-2)}D_{q,2}M^{3/2}\|f\|_\infty,  \qquad  Y\le
\sigma_{r-2}^{-2}D_{q,r-2}M^{3/2}\|f\|_\infty. 
\end{equation*}
The two-block case of Lemma~\ref{lem:local-complement-interpolation} yields 
\begin{equation*}
\|a_{2,r-2}\|_{p_r}  \le X^{2/r}Y^{(r-2)/r}  \le
R_{q,r}D_{q,2}^{2/r}D_{q,r-2}^{(r-2)/r}  M^{3/2}\|f\|_\infty, 
\end{equation*}
where 
\begin{equation*}
R_{q,r}:=  \sigma_2^{-2(r-2)/r}\sigma_{r-2}^{-2(r-2)/r}. 
\end{equation*}
The probability that exactly two of the $r$ active coordinates fall in the
first block is $\pi_r=2^{-r}\binom r2$. Averaging the $p_r$-th power gives %
\eqref{eq:fixed-local-levels} with 
\begin{equation*}
D_{q,r}:=\pi_r^{-1/p_r}R_{q,r}  D_{q,2}^{2/r}D_{q,r-2}^{(r-2)/r}. 
\end{equation*}
This completes the induction.
\end{proof}

\subsection{A fixed-block weighted estimate}

\label{sec:fixed-block-local} For $L\ge2$ put 
\begin{equation}  \label{eq:local-lambda-b}
\lambda_L:=\frac{L-1}{L},\qquad b_{q,L}:=\frac{2\gamma_q}{(L-1)\log(L/(L-1))}%
.
\end{equation}

The next lemma adapts the simultaneous row and column estimate of
\attribrev{Slote, Tseng and Volberg} \cite[%
Lemma~5.2]{SloteTsengVolberg} to the local system $\Phi$. Keeping the sum
over complementary degrees is essential: a single bidegree projection can
have exponentially larger supremum norm. Counterexamples communicated by
Lars Becker and Haonan Zhang are recorded in \cite[Appendix~B]%
{SloteTsengVolberg}.

\begin{lemma}
\label{lem:simultaneous-local-sections} Let $w(1),\ldots,w(M)>0$. Suppose
that 
\begin{equation*}
\left(\sum_{s=1}^M\frac{\mathcal{A}_s(h)^2}{w(s)^2}\right)^{1/2}  \le
H\|h\|_\infty 
\end{equation*}
for every function $h$ of interaction order at most $M$ on at most $N_0$
coordinates. Fix a partition $I\dot\cup J$ of at most $N_0$ coordinates and
a function $Q:\Omega^I\times\Omega^J\to\mathbb{C}$ of interaction order at
most $M$. Write 
\begin{equation*}
Q(x,y)=\sum_{\alpha,\xi}b_{\alpha,\xi}\Phi_\alpha(x)\Phi_\xi(y),  \qquad 
Y_{s,t}:=\left[\sum_{s(\alpha)=s} 
\left(\sum_{s(\xi)=t}|b_{\alpha,\xi}|^2\right)^{p_s/2}  \right]^{1/p_s}. 
\end{equation*}
Here $1\le s\le M$ and $0\le t\le M-s$ are integers.
Then 
\begin{equation}  \label{eq:simultaneous-local-sections}
\sum_{s=2}^M\sum_{t=0}^{M-s} \frac{\sigma_s^{2t}Y_{s,t}^2}{w(s)^2} \le
H^2\|Q\|_\infty^2.
\end{equation}
\end{lemma}

\begin{proof}
For each complement index $\alpha$, put $g_\alpha(y)=\sum_\xi
b_{\alpha,\xi}\Phi_\xi(y)$. Fix $s\ge2$ and set $p=p_s\le2$.
Since $2/p\ge1$, Minkowski applies to the nonnegative array
$A_{\alpha,t}=\sigma_s^{2t}\sum_{s(\xi)=t}|b_{\alpha,\xi}|^2$.
The sum over $0\le t\le M-s$ contains the full complementary interaction
spectrum of $g_\alpha$, rather than an individual bidegree projection.
Minkowski's inequality, with exponent $2/p$, gives 
\begin{align*}
\left(\sum_{t=0}^{M-s}\sigma_s^{2t}Y_{s,t}^2\right)^{1/2} &\le\left[\sum_{s(\alpha)=s}
\left(\sum_{t=0}^{M-s}\sigma_s^{2t} \sum_{s(\xi)=t}|b_{\alpha,\xi}|^2\right)^{p/2} %
\right]^{1/p} \\
&=\left(\sum_{s(\alpha)=s} \|T_{\sigma_s}^{\otimes
J}g_\alpha\|_2^p\right)^{1/p} \\
&\le\left(\int_{\Omega^J}
\sum_{s(\alpha)=s}|g_\alpha(y)|^p\,d\mu_q^J(y)\right)^{1/p} \\
&=\|\mathcal{A}_s(Q(\cdot,y))\|_{L^p(y)} \le\|\mathcal{A}_s(Q(\cdot,y))%
\|_{L^2(y)}.
\end{align*}
The equality uses Parseval and the next inequality uses Potts
hypercontractivity. Divide the squared inequality by $w(s)^2$, sum over $s$,
and apply the hypothesis to each section $Q(\cdot,y)$. This gives 
\begin{equation*}
\sum_{s=2}^M\sum_{t=0}^{M-s}  \frac{\sigma_s^{2t}Y_{s,t}^2}{w(s)^2} 
\le\int_{\Omega^J}\sum_{s=2}^M  \frac{\mathcal{A}_s(Q(\cdot,y))^2}{w(s)^2}%
\,d\mu_q^J(y)  \le H^2\|Q\|_\infty^2. 
\end{equation*}
\end{proof}

\attribrev{The next proposition adapts the fixed-weight contraction argument
of \cite[Section~2]{Ivanisvili2026} and
\cite[Sections~5--8]{SloteVolbergPolynomial} to $L$ coordinate blocks.}

\begin{proposition}
\label{prop:fixed-block-local} Let $L\ge2$ and $B>b_{q,L}$. There is $%
K_{q,L,B}<\infty$ such that, for all $M,N\ge1$ and every $f:\Omega^N\to%
\mathbb{C}$ of interaction order at most $M$, 
\begin{equation}  \label{eq:fixed-block-local}
\left(\sum_{r=1}^M\frac{\mathcal{A}_r(f)^2}{r^{2B}}\right)^{1/2} \le
K_{q,L,B}M^{3/2}\|f\|_\infty.
\end{equation}
The constant is independent of $N$ and $\Phi$.
\end{proposition}

\begin{proof}
\lastrev{Fix an upper bound $N_0$ on the dimension} and let $\Gamma_{M,N_0}$ be the least constant
in \eqref{eq:fixed-block-local} over all local orthonormal bases and
dimensions at most $N_0$, with the factor $M^{3/2}$ removed from the
right-hand side. Equivalently, 
\begin{equation*}
\Gamma_{M,N_0}:=  \sup\frac{\left(\sum_{r=1}^M\mathcal{A}_r(f)^2/r^{2B}%
\right)^{1/2}} {\|f\|_\infty}. 
\end{equation*}
\lastrev{For $1\le r\le\min\{M,N_0\}$, H\"older's inequality and Parseval give
\[
\mathcal A_r(f)
\le\left[\binom{N_0}{r}(q-1)^r\right]^{1/(2r)}\|f\|_2
\le\left[\binom{N_0}{r}(q-1)^r\right]^{1/(2r)}\|f\|_\infty.
\]
Consequently,
\[
\Gamma_{M,N_0}^2
\le\sum_{r=1}^{\min\{M,N_0\}}
\frac{\left[\binom{N_0}{r}(q-1)^r\right]^{1/r}}{r^{2B}}
<\infty,
\]
uniformly over the local orthonormal basis.}

Partition the coordinates into $L$ blocks and let $Q$ denote the regrouped
function. For a multidegree $\mathbf{d}$ with $|\mathbf{d}|\le M$, put $s_j=|%
\mathbf{d}|-d_j$ and, when $s_j\ge2$, define 
\begin{equation}  \label{eq:local-Uj-fixed}
U_j^{(B)}(\mathbf{d}):=\frac{\sigma_{s_j}^{d_j}X_j(a_\mathbf{d}(Q))}{s_j^B}.
\end{equation}
For each $j$, 
\begin{equation}  \label{eq:local-simultaneous-fixed}
\sum_{\substack{ \mathbf{d}:\,|\mathbf{d}|\le M \\ s_j\ge2}}U_j^{(B)}(\mathbf{d}%
)^2 \le \Gamma_{M,N_0}^2\|Q\|_\infty^2.
\end{equation}
To verify this, take $J=I_j$ and $I=\bigcup_{k\ne j}I_k$ in Lemma~\ref%
{lem:simultaneous-local-sections}. Fix $s=s_j$ and $t=d_j$. The complement
indices $\alpha$ with $s(\alpha)=s$ are partitioned by their interaction
orders in the $L-1$ blocks of $I$. Thus the corresponding mixed norms
satisfy 
\begin{equation*}
\sum_{\substack{ \mathbf{d}:\,s_j=s,\ d_j=t}}  X_j(a_\mathbf{d}%
(Q))^{p_s}=Y_{s,t}^{p_s}. 
\end{equation*}
Since $p_s\le2$, 
\begin{equation*}
\sum_{\substack{ \mathbf{d}:\,s_j=s,\ d_j=t}}  X_j(a_\mathbf{d}(Q))^2\le
Y_{s,t}^2. 
\end{equation*}
Multiply by $\sigma_s^{2t}/s^{2B}$ and sum over $s,t$. Lemma~\ref%
{lem:simultaneous-local-sections}, with $w(s)=s^B$ and $H=\Gamma_{M,N_0}$,
gives \eqref{eq:local-simultaneous-fixed}.

Assign every coordinate independently and uniformly to one of the $L$
blocks. \introrev{For sufficiently large $r$ that $s_j\ge2$ on the set below,
consider the multidegrees in} 
\begin{equation}  \label{eq:local-fixed-window}
\mathcal{D}_r:=\left\{\mathbf{d}\in\mathbb{N}_0^L:\ |\mathbf{d}|=r, \
|d_j-r/L|\le r^{3/4}\ (1\le j\le L)\right\}.
\end{equation}
\lastrev{For a coefficient of $f$ with support of cardinality $r$}, the vector of
block interaction orders is multinomial $\func{Mult}(r;1/L,\ldots,1/L)$; the
local labels remain attached to their coordinates and no coefficients merge.
If 
\begin{equation*}
{\eta^{\mathrm{fix}}_{r,L}}:=L^{-r}\sum_{\mathbf{d}\in\mathcal{D}_r}\frac{r!}{d_1!\cdots d_L!}%
,  \qquad {k^{\mathrm{fix}}_{r,L}}:=|\mathcal{D}_r|, 
\end{equation*}
then 
\begin{equation}  \label{eq:local-fixed-capture}
\int\sum_{\mathbf{d}\in\mathcal{D}_r}\|a_\mathbf{d}(Q_\omega)\|_{p_r}^{p_r}
\,d\omega ={\eta^{\mathrm{fix}}_{r,L}}\,\mathcal{A}_r(f)^{p_r}.
\end{equation}
Moreover 
\begin{equation}  \label{eq:local-fixed-prob-count}
{\eta^{\mathrm{fix}}_{r,L}}\ge1-\frac{1-1/L}{\sqrt r},\qquad {k^{\mathrm{fix}}_{r,L}}\le(r+1)^{L-1}.
\end{equation}
\introrev{
For a fixed support $S\subseteq\{1,\ldots,N\}$ with $|S|=r$, put
\[
Z_j(\omega):=\sum_{k\in S}\mathbf1_{\{\omega(k)=j\}}
\qquad(\omega\in\Pi_{N,L},\ 1\le j\le L).
\]
Independence of the coordinate assignments gives
\begin{align*}
\int_{\Pi_{N,L}}Z_j\,d\nu_{N,L}&=\frac rL,\\
\int_{\Pi_{N,L}}(Z_j-r/L)^2\,d\nu_{N,L}
&=\sum_{k\in S}\int_{\Pi_{N,L}}
(\mathbf1_{\{\omega(k)=j\}}-1/L)^2\,d\nu_{N,L}\\
&=\frac rL\left(1-\frac1L\right).
\end{align*}
The cross terms vanish because each centered indicator has integral zero.
Therefore
\[
\int_{\Pi_{N,L}}\sum_{j=1}^L(Z_j-r/L)^2\,d\nu_{N,L}
=r\left(1-\frac1L\right).
\]
If $(Z_1,\ldots,Z_L)\notin\mathcal D_r$, this sum exceeds $r^{3/2}$.
Markov's inequality consequently gives
\[
1-\eta^{\mathrm{fix}}_{r,L}
\le\frac{r(1-1/L)}{r^{3/2}}
=\frac{1-1/L}{\sqrt r},
\]
which proves the first estimate in \eqref{eq:local-fixed-prob-count}.
}

For $\mathbf{d}\in\mathcal{D}_r$ put $\theta_j=s_j/((L-1)r)$ and 
\begin{equation}  \label{eq:local-fixed-R}
R_q(\mathbf{d}):=\prod_{j=1}^L\sigma_{s_j}^{-d_j\theta_j}.
\end{equation}
Lemma~\ref{lem:local-complement-interpolation} and \eqref{eq:local-Uj-fixed}
give 
\begin{equation}  \label{eq:local-fixed-profile}
\frac{\|a_\mathbf{d}(Q)\|_{p_r}}{r^B} \le R_q(\mathbf{d})\frac{%
\introrev{\prod_{j=1}^L}s_j^{B\theta_j}}{r^B} \introrev{\prod_{j=1}^L}U_j^{(B)}(\mathbf{d})^{\theta_j}.
\end{equation}
Define 
\begin{align}
\mathcal K_q(r,\mathbf{d})&:=\bigl(\eta^{\mathrm{fix}}_{r,L}\bigr)^{-1/p_r}\bigl(k^{\mathrm{fix}}_{r,L}\bigr)^{1/(2r)} R_q(\mathbf{d})\frac{%
\introrev{\prod_{j=1}^L}s_j^{B\theta_j}}{r^B},  \label{eq:local-fixed-K} \\
\kappa_q(r,\mathbf{d})&:=\sqrt{L\introrev{\max_{1\le j\le L}}\theta_j}\,\mathcal K_q(r,\mathbf{d}).
\label{eq:local-fixed-kappa}
\end{align}
For each $r$, Jensen's inequality in the partition variable and the
comparison of the finite $\ell^{p_r}$ and $\ell^2$ norms give 
\begin{equation*}
\frac{\mathcal{A}_r(f)^2}{r^{2B}}  \le\bigl(\eta^{\mathrm{fix}}_{r,L}\bigr)^{-2/p_r}\bigl(k^{\mathrm{fix}}_{r,L}\bigr)^{1/r} 
\int\sum_{\mathbf{d}\in\mathcal{D}_r}  \frac{\|a_\mathbf{d}%
(Q_\omega)\|_{p_r}^2}{r^{2B}}\,d\omega. 
\end{equation*}
By \eqref{eq:local-fixed-profile} and weighted arithmetic--geometric mean,
each summand, including its scalar factor, is at most 
\begin{equation*}
\mathcal K_q(r,\mathbf{d})^2\introrev{\prod_{j=1}^L}U_j^{(B)}(\mathbf{d})^{2\theta_j}  \le \mathcal K_q(r,\mathbf{d}%
)^2\introrev{\sum_{j=1}^L}\theta_jU_j^{(B)}(\mathbf{d})^2  \le\frac{\kappa_q(r,\mathbf{d})^2}{L}%
\introrev{\sum_{j=1}^L}U_j^{(B)}(\mathbf{d})^2. 
\end{equation*}
Consequently, if $\kappa_q(r,\mathbf{d})\le c<1$ for every $r\ge R$ and $%
\mathbf{d}\in\mathcal{D}_r$, summing first over $r,\mathbf{d}$ and then
using \eqref{eq:local-simultaneous-fixed} yields 
\begin{equation}  \label{eq:local-fixed-high}
\sum_{r=R}^M\frac{\mathcal{A}_r(f)^2}{r^{2B}} \le\frac{c^2}{L}\sum_{j=1}^L
\int\sum_{\substack{ \mathbf{d}:\,|\mathbf{d}|\le M \\ s_j\ge2}}U_j^{(B)}(\mathbf{d%
})^2\,d\omega \le c^2\Gamma_{M,N_0}^2\|f\|_\infty^2.
\end{equation}

It remains to verify the strict multiplier bound. Put $\lambda=\lambda_L$, $%
a_r=r^{-1/4}$ and $v_r=[r^2(\lambda-a_r)^2]^{-1}$. On %
\eqref{eq:local-fixed-window}, 
\begin{equation*}
\lambda-a_r\le \frac{s_j}{r}\le\lambda+a_r,  \qquad  L\introrev{\max_{1\le j\le L}}\theta_j\le1+%
\frac{a_r}{\lambda}. 
\end{equation*}
Since 
\begin{equation*}
s\log\frac{s+1}{s-1}  \le \frac{2}{1-s^{-2}},\qquad s>1, 
\end{equation*}
we have 
\begin{equation}  \label{eq:local-fixed-noise}
\log R_q(\mathbf{d}) =\frac{\gamma_q}{(L-1)r} \introrev{\sum_{j=1}^L}d_js_j\log\frac{s_j+1}{%
s_j-1} \le\frac{2\gamma_q}{(L-1)(1-v_r)}.
\end{equation}
Also 
\begin{equation}  \label{eq:local-fixed-degree-weight}
\log\frac{\introrev{\prod_{j=1}^L}s_j^{B\theta_j}}{r^B} =B\introrev{\sum_{j=1}^L}\theta_j\log(s_j/r) \le
B\log\lambda+B\log(1+a_r/\lambda).
\end{equation}
Set 
\begin{equation}  \label{eq:local-fixed-eta}
\eta_{q,L,B}:=B\log(1/\lambda)-\frac{2\gamma_q}{L-1}>0.
\end{equation}
Using \eqref{eq:local-fixed-prob-count}--\eqref{eq:local-fixed-degree-weight}
in \eqref{eq:local-fixed-kappa}, for all sufficiently large $r$ we obtain
\begin{equation}\label{eq:fixed-block-explicit-error}
\log\kappa_q(r,\mathbf d)\le-\eta_{q,L,B}+E_{q,L,B}(r),
\end{equation}
where, with $a_r=r^{-1/4}$, $\lambda=(L-1)/L$, and
$v_r=[r^2(\lambda-a_r)^2]^{-1}$,
\begin{align*}
E_{q,L,B}(r):={}&\left(B+\frac12\right)
 \log\left(1+\frac{a_r}{\lambda}\right)
 +\frac{2\gamma_q}{L-1}\frac{v_r}{1-v_r}\\
&+\frac{L-1}{2r}\log(r+1)
 -\frac{r+1}{2r}\log\left(1-\frac{1-1/L}{\sqrt r}\right).
\end{align*}
Indeed, the first logarithm bounds together the degree-weight
correction and $\frac12\log(L\max_j\theta_j)$, the second term
is the excess over $2\gamma_q/(L-1)$ in
\eqref{eq:local-fixed-noise}, while the last two terms bound the
multidegree-count and capture factors. Each term tends to zero for
fixed $q,L,B$ as $r\to\infty$, and none depends on $\mathbf d$.
Consequently, some $R=R(q,L,B)$ satisfies $s_j\ge2$ on
$\mathcal D_r$ and $E_{q,L,B}(r)\le\eta_{q,L,B}/2$ whenever
$r\ge R$, so that
\begin{equation*}
\kappa_q(r,\mathbf d)\le c:=e^{-\eta_{q,L,B}/2}<1.
\end{equation*}

For the Boolean choice $q=2$, $L=3$, $B=5/4$, this contraction
can be certified without a limiting argument or numerical estimates.
Put $\mu=\frac54\log(3/2)-\frac12$. The positive-term expansion
\begin{equation*}
\log\frac32
 =2\sum_{n=0}^{\infty}\frac{5^{-(2n+1)}}{2n+1}
 >2\left(\frac15+\frac1{375}\right)=\frac{152}{375}
\end{equation*}
shows that $\mu>1/150$. For $r\ge10^{12}$ we have
$a_r\le1/1000$, $\lambda-a_r>1/2$, and
$v_r\le4r^{-2}<1/2$. The elementary inequalities
$\log(1+x)\le x$ for $x\ge0$,
$-\log(1-x)\le2x$ for $0\le x\le1/2$, and
$\log(1+r)\le\sqrt r$ for $r>0$ yield
\begin{align*}
E_{2,3,5/4}(r)
&\le\frac{21}{8}r^{-1/4}
   +\frac4{r^2}+\frac1{\sqrt r}+\frac4{3\sqrt r}\\
&\le\frac{21}{8000}+\frac7{3\cdot10^6}
   +\frac4{10^{24}}<\frac1{300}.
\end{align*}
The latter inequality compares only rational numbers. Thus
\begin{equation}\label{eq:fixed-block-boolean-explicit-contraction}
\log\kappa_2(r,\mathbf d)
 <-\frac1{150}+\frac1{300}=-\frac1{300},
\qquad r\ge10^{12},\quad \mathbf d\in\mathcal D_r.
\end{equation}
In this particular case one may therefore take $R=10^{12}$ and
$c=e^{-1/300}$ in \eqref{eq:local-fixed-high}.

For $r<R$, Proposition~\ref{prop:fixed-local-levels} gives 
\begin{equation*}
\left(\sum_{r=1}^{\min\{M,R-1\}}\frac{\mathcal{A}_r(f)^2}{r^{2B}}\right)^{1/2}  \le
H_{q,L,B}M^{3/2}\|f\|_\infty 
\end{equation*}
with $H_{q,L,B}<\infty$. Combining this estimate with %
\eqref{eq:local-fixed-high} and taking the supremum in the definition of $%
\Gamma_{M,N_0}$ yields 
\begin{equation*}
\Gamma_{M,N_0}^2  \le H_{q,L,B}^2M^3+c^2\Gamma_{M,N_0}^2. 
\end{equation*}
Thus 
\begin{equation*}
\Gamma_{M,N_0}\le \frac{H_{q,L,B}}{\sqrt{1-c^2}}M^{3/2}. 
\end{equation*}
The bound is independent of $N_0$; letting $N_0\to\infty$ proves %
\eqref{eq:fixed-block-local}.
\end{proof}

\section{Proof of Theorem A}
\label{sec:proof-A}

\introrev{We adapt the weighted contraction argument of Slote, Tseng and Volberg
\cite[Sections~5--8]{SloteTsengVolberg}, using
Lemma~\ref{lem:simultaneous-local-sections}, the weight $W_q$ below and
a number of coordinate blocks of order $\log M$.}

\subsection{The critical degree weight}

\label{sec:critical-weight} We use $L\ge3$ coordinate blocks and let their
\lastrev{number grow with the bound $M$ on the interaction order.}

Fix $0<a<1/20$. For an integer $K_q\ge1$, to be chosen below, set 
\begin{equation}  \label{eq:critical-local-weight}
W_q(r):=r^{2\gamma_q}(1+\log r)^{K_q}\qquad(r\ge1).
\end{equation}
For $M,N\ge1$, let $\Gamma_{M,N}^{(q)}$ be the least constant such that 
\begin{equation}  \label{eq:critical-local-Gamma}
\left(\sum_{r=1}^M\frac{\mathcal{A}_r(f)^2}{W_q(r)^2}\right)^{1/2} \le
\Gamma_{M,N}^{(q)}\|f\|_\infty
\end{equation}
for every local orthonormal basis and every function of dimension at most $N$
and interaction order at most $M$. Parseval and coefficient counting make
this supremum finite.

\introrev{
Let $\mathcal I$ be a nonempty finite index set. For each
$\ell\in\mathcal I$, let
$\theta_\ell=(\theta_{1,\ell},\ldots,\theta_{L,\ell})$ have positive entries
with $\sum_{j=1}^L\theta_{j,\ell}=1$.
If $u_{j,\ell}\ge0$ satisfy $\sum_{\ell\in\mathcal I}u_{j,\ell}\le1$
for $1\le j\le L$, then} 
\begin{equation}  \label{eq:local-entropy-sum}
\introrev{\sum_{\ell\in\mathcal I}}\prod_{j=1}^Lu_{j,\ell}^{\theta_{j,\ell}} \le \exp\left(\introrev{\sup_{\ell\in\mathcal I}}
D(\theta_\ell\|\mathbf{u}_L)\right),
\end{equation}
where $\mathbf{u}_L=(1/L,\ldots,1/L)$ and $D(\theta\|\mathbf{u}%
_L)=\introrev{\sum_{j=1}^L}\theta_j\log(L\theta_j)$. Indeed, concavity of the logarithm gives 
\begin{equation*}
\introrev{\prod_{j=1}^L} u_{j,\ell}^{\theta_{j,\ell}}  \le e^{D(\theta_\ell\|\mathbf{u}%
_L)}\frac1L\introrev{\sum_{j=1}^L}u_{j,\ell}, 
\end{equation*}
and summation in $\ell$ proves \eqref{eq:local-entropy-sum}.

Put 
\begin{equation}  \label{eq:critical-local-window}
\mathcal{C}_r(L):= \left\{\mathbf{d}\in\mathbb{N}_0^L:\ |\mathbf{d}|=r,\quad
\left|\frac{d_j}{r}-\frac1L\right|\le\frac aL \ (1\le j\le L)\right\}.
\end{equation}
For every coefficient with support size $r$, a uniform random assignment of
coordinates to the $L$ blocks gives the multinomial distribution $\func{Mult}%
(r;1/L,\ldots,1/L)$ for the block interaction orders. Thus, with 
\begin{equation*}
\eta_{r,L}:=L^{-r}\sum_{\mathbf{d}\in\mathcal{C}_r(L)}\frac{r!}{d_1!\cdots
d_L!}, 
\end{equation*}
Bernstein's inequality \cite[Section~2.8]{BoucheronLugosiMassart} and a
union bound give $c,C>0$, depending only on the fixed parameter $a$, such that 
\begin{equation}  \label{eq:critical-local-capture}
\eta_{r,L}\ge1-CL\exp(-cr/L).
\end{equation}
For $\mathbf{d}\in\mathcal{C}_r(L)$ put 
\begin{equation*}
y_j:=\frac{s_j}{r},\qquad  \theta_j:=\frac{s_j}{(L-1)r}. 
\end{equation*}
Writing $\lambda=(L-1)/L$ and $\delta_j=d_j/r-1/L$, one has 
\begin{equation}  \label{eq:critical-local-balance}
y_j=\lambda-\delta_j,\qquad |\delta_j|\le a/L,\qquad \introrev{\sum_{j=1}^L}\delta_j=0,
\end{equation}
and $\theta_j-1/L=-\delta_j/(L-1)$. The inequality $\log u\le u-1$ gives 
\begin{align*}
D(\theta\|\mathbf{u}_L) &\le\introrev{\sum_{j=1}^L}\theta_j(L\theta_j-1)
=L\introrev{\sum_{j=1}^L}(\theta_j-1/L)^2 \\
&=\frac{L}{(L-1)^2}\introrev{\sum_{j=1}^L}\delta_j^2 \le\frac{a^2}{(L-1)^2}.
\end{align*}
In particular, 
\begin{equation}  \label{eq:critical-local-entropy}
D(\theta\|\mathbf{u}_L)\le \rev{\frac{a^2}{(L-1)^2}}.
\end{equation}

For a regrouped function $Q$ of dimension at most $N$ and interaction
order at most $M$, and a multidegree $\mathbf d$ with
$|\mathbf d|\le M$ and $s_j\ge2$, define 
\begin{equation}  \label{eq:critical-local-Uj}
U_j(\mathbf{d}):= \frac{\sigma_{s_j}^{d_j}X_j(a_\mathbf{d}(Q))}{W_q(s_j)}.
\end{equation}
Apply Lemma~\ref{lem:simultaneous-local-sections} with $w(s)=W_q(s)$ and $%
H=\Gamma_{M,N}^{(q)}$. Partitioning the complement indices as in the proof
of \eqref{eq:local-simultaneous-fixed} gives, for each $j$, 
\begin{equation}  \label{eq:critical-local-simultaneous}
\sum_{\substack{ \mathbf{d}:\,|\mathbf{d}|\le M \\ s_j\ge2}}U_j(\mathbf{d}%
)^2 \le (\Gamma_{M,N}^{(q)})^2\|Q\|_\infty^2.
\end{equation}
By Lemma~\ref{lem:local-complement-interpolation}, 
\begin{equation}  \label{eq:critical-local-profile}
\frac{\|a_\mathbf{d}(Q)\|_{p_r}}{W_q(r)} \le \mathcal{M}_q(r,\mathbf{d})
\prod_{j=1}^LU_j(\mathbf{d})^{\theta_j},
\end{equation}
where 
\begin{equation}  \label{eq:critical-local-multiplier}
\mathcal{M}_q(r,\mathbf{d}):= \prod_{j=1}^L\sigma_{s_j}^{-d_j\theta_j}
\prod_{j=1}^L\left(\frac{W_q(s_j)}{W_q(r)}\right)^{\theta_j}.
\end{equation}

We estimate the multiplier for $r\ge L^4$. Define 
\begin{equation*}
h_q(s):=\gamma_q s\log\frac{s+1}{s-1}. 
\end{equation*}
The elementary integral estimate used in \eqref{eq:local-fixed-noise} gives 
\begin{equation*}
2\gamma_q\le h_q(s)\le\frac{2\gamma_q}{1-s^{-2}}. 
\end{equation*}
Since $s_j=ry_j\ge r/2$ on \eqref{eq:critical-local-window}, 
\begin{align}
-\sum_{j=1}^Ld_j\theta_j\log\sigma_{s_j}
&=\frac1{(L-1)r}\sum_{j=1}^Ld_jh_q(s_j)  \notag \\
&\le \frac{2\gamma_q}{L-1}+\frac{C(q)}{L^9}.  \label{eq:critical-local-noise}
\end{align}
For the power part of $W_q$, the identity $L\theta_j=y_j/\lambda$ gives 
\begin{equation*}
\introrev{\sum_{j=1}^L}\theta_j\log y_j  =\log\lambda+D(\theta\|\mathbf{u}_L). 
\end{equation*}
Thus \eqref{eq:critical-local-entropy} implies 
\begin{equation}  \label{eq:critical-local-power}
2\gamma_q\sum_{j=1}^L\theta_j\log y_j \le 2\gamma_q\log\lambda
+\rev{\frac{2\gamma_q a^2}{(L-1)^2}}.
\end{equation}
Put $T:=1+\log M$ and $t:=1+\log r\le T$. Since $y_j\le1-(1-a)/L$, we have $%
\log y_j\le-(1-a)/L$. Also 
\begin{equation*}
\log(1+\log s_j)-\log(1+\log r)  =\log\left(1+\frac{\log y_j}{t}\right)  \le%
\frac{\log y_j}{t}  \le-\frac{1-a}{LT}. 
\end{equation*}
Multiplying by $K_q\theta_j$ and summing gives 
\begin{align}
&K_q\sum_{j=1}^L\theta_j \bigl[\log(1+\log s_j)-\log(1+\log r)\bigr]  \notag
\\
&\qquad\le -\frac{(1-a)K_q}{LT}.  \label{eq:critical-local-logweight}
\end{align}
Combining \eqref{eq:critical-local-noise}--%
\eqref{eq:critical-local-logweight}, 
\begin{equation}  \label{eq:critical-local-multiplier-bound}
\log\mathcal{M}_q(r,\mathbf{d}) \le \frac{2\gamma_q}{L-1} +2\gamma_q\log%
\frac{L-1}{L} -\frac{(1-a)K_q}{LT}
+\rev{\frac{2\gamma_q a^2}{(L-1)^2}+\frac{C(q)}{L^9}}.
\end{equation}
For a general fixed power $r^\eta$, the leading loss takes the form
\begin{equation}
\frac{2\gamma_q}{L-1}
+\eta\log\frac{L-1}{L}
=\frac{2\gamma_q-\eta}{L}+O_{q,\eta}(L^{-2}).
\label{eq:critical-general-power}
\end{equation}
Thus the choice $\eta=2\gamma_q$ cancels the term of order $L^{-1}$:
\begin{equation}
\frac{2\gamma_q}{L-1}
+2\gamma_q\log\frac{L-1}{L}
=2\gamma_q\bigl(u-\log(1+u)\bigr)=O_q(L^{-2}),
\qquad u:=\frac1{L-1}.
\label{eq:critical-leading-cancellation}
\end{equation}
\introrev{With $L\asymp\delta T$, the logarithmic term in
\eqref{eq:critical-local-logweight} has order $T^{-2}$, as does the
remaining positive error. After fixing $a$ and $\delta$, we choose $K_q$
so that the multiplier is less than one.}
A smaller fixed power leaves a positive term of order $L^{-1}$,
which this logarithmic correction cannot absorb.

The averaging identity corresponding to \eqref{eq:critical-local-capture} is 
\begin{equation}  \label{eq:critical-local-capture-identity}
\int\sum_{\mathbf{d}\in\mathcal{C}_r(L)} \|a_\mathbf{d}(Q_\omega)%
\|_{p_r}^{p_r}\,d\omega =\eta_{r,L}\mathcal{A}_r(f)^{p_r}.
\end{equation}
Let $k_{r,L}=|\mathcal{C}_r(L)|$. Since the first $L-1$ entries determine
the last, 
\begin{equation}  \label{eq:critical-local-cardinality}
k_{r,L}\le(r+1)^{L-1},\qquad \frac{\log k_{r,L}}{2r}\le C\frac{\log L}{L^3}
\quad(r\ge L^4).
\end{equation}
Moreover \eqref{eq:critical-local-capture} gives 
\begin{equation}  \label{eq:critical-local-probability-loss}
-\frac1{p_r}\log \eta_{r,L}\le CL\exp(-cL^3) \quad(r\ge L^4)
\end{equation}
for all sufficiently large $L$.

\lastrev{We now sum over the multidegrees in $\mathcal C_r(L)$.} For a
fixed $r$, set 
\begin{equation*}
Z_{r,\omega}:=  \sum_{\mathbf{d}\in\mathcal{C}_r(L)}  \left(\frac{\|a_%
\mathbf{d}(Q_\omega)\|_{p_r}}{W_q(r)}\right)^{p_r}. 
\end{equation*}
Since $2/p_r=1+1/r$, Jensen's inequality in $\omega$ and the
finite-dimensional comparison
$\|z\|_{p_r}\le k_{r,L}^{1/(2r)}\|z\|_2$ give 
\begin{align}
\frac{\mathcal{A}_r(f)^2}{W_q(r)^2} &=\eta_{r,L}^{-2/p_r}\left(\int
Z_{r,\omega}\,d\omega\right)^{2/p_r}  \notag \\
&\le \eta_{r,L}^{-2/p_r} \int Z_{r,\omega}^{2/p_r}\,d\omega  \notag \\
&\le \eta_{r,L}^{-2/p_r}k_{r,L}^{1/r} \int\sum_{\mathbf{d}\in\mathcal{C}%
_r(L)} \left(\frac{\|a_\mathbf{d}(Q_\omega)\|_{p_r}}{W_q(r)}\right)^2d\omega.
\label{eq:critical-local-holder-open}
\end{align}
Insert \eqref{eq:critical-local-profile}. For $f\ne0$ and each partition $%
\omega$ put 
\begin{equation*}
G_\omega:=\Gamma_{M,N}^{(q)}\|Q_\omega\|_\infty 
=\Gamma_{M,N}^{(q)}\|f\|_\infty,  \qquad  u_{j,(r,\mathbf{d})}:=\frac{U_j(%
\mathbf{d})^2}{G_\omega^2}. 
\end{equation*}
By \eqref{eq:critical-local-simultaneous}, for every $j$, 
\begin{equation*}
\sum_{r=1}^M\sum_{\substack{ \mathbf{d}:\,|\mathbf{d}|=r \\ s_j\ge2}} 
u_{j,(r,\mathbf{d})}\le1. 
\end{equation*}
Thus the index $\ell$ in \eqref{eq:local-entropy-sum} may be taken to be the
pair $(r,\mathbf{d})$.
Apply \eqref{eq:local-entropy-sum} to
$\mathcal I=\{(r,\mathbf d):L^4\le r\le M,\ \mathbf d\in\mathcal C_r(L)\}$.
For each partition $\omega$, \eqref{eq:critical-local-simultaneous}
and \eqref{eq:critical-local-entropy} give
\begin{equation*}
\sum_{r=L^4}^{M}\sum_{\mathbf d\in\mathcal C_r(L)}
\prod_{j=1}^{L} U_j(\mathbf d)^{2\theta_j}
\le \exp\!\left(\frac{a^2}{(L-1)^2}\right)
\bigl(\Gamma_{M,N}^{(q)}\bigr)^2\|f\|_\infty^2.
\end{equation*}
Sum \eqref{eq:critical-local-holder-open} over $r\ge L^4$
and take the supremum of the scalar multiplier. The preceding estimate
gives 
\begin{align}
&\sum_{r=L^4}^M\frac{\mathcal{A}_r(f)^2}{W_q(r)^2}  \notag \\
&\quad\le \exp\!\left(\rev{\frac{a^2}{(L-1)^2}}\right) \sup_{\substack{ L^4\le r\le M
\\ \mathbf{d}\in\mathcal{C}_r(L)}} \left[ \eta_{r,L}^{-2/p_r}k_{r,L}^{1/r} 
\mathcal{M}_q(r,\mathbf{d})^2 \right] (\Gamma_{M,N}^{(q)})^2\|f\|_\infty^2.
\label{eq:critical-local-entropy-closed}
\end{align}
Consequently 
\begin{equation}  \label{eq:critical-local-high}
\left(\sum_{r=L^4}^M\frac{\mathcal{A}_r(f)^2}{W_q(r)^2}\right)^{1/2} \le
\kappa_{q,M} \Gamma_{M,N}^{(q)}\|f\|_\infty,
\end{equation}
where one may take 
\begin{equation*}
\kappa_{q,M}:=  \exp\!\left(\rev{\frac{a^2}{2(L-1)^2}}\right)  \sup_{\substack{ L^4\le
r\le M \\ \mathbf{d}\in\mathcal{C}_r(L)}} 
\eta_{r,L}^{-1/p_r}k_{r,L}^{1/(2r)}\mathcal{M}_q(r,\mathbf{d}). 
\end{equation*}
Equations \eqref{eq:critical-local-multiplier-bound}, %
\eqref{eq:critical-local-cardinality}, and %
\eqref{eq:critical-local-probability-loss} therefore yield 
\begin{align}
\log\kappa_{q,M} \le{}& \frac{2\gamma_q}{L-1} +2\gamma_q\log\frac{L-1}{L} -%
\frac{(1-a)K_q}{LT}
+\rev{\frac{(2\gamma_q+1/2)a^2}{(L-1)^2}}  \notag \\
&+\rev{\frac{C(q)}{L^9}}+C\frac{\log L}{L^3}+CL\exp(-cL^3).  \label{eq:critical-local-total-loss}
\end{align}

Choose 
\begin{equation}  \label{eq:critical-local-L}
L:=\max\{3,\lfloor\delta T\rfloor\}
\end{equation}

with a fixed $\delta>0$, chosen before $K_q$. For sufficiently large $M$,
$L\sim\delta T$ and $L^4\le M$. Moreover, on $\mathcal C_r(L)$ with
$r\ge L^4$, one has $s_j\ge r/2\ge2$, so all radii used above are defined.
Absorbing the positive terms of \eqref{eq:critical-local-total-loss}
into a constant $C^*_{q,a}$ independent of $K_q$, we obtain
\begin{equation*}
\log\kappa_{q,M}\le
\left(\frac{C^*_{q,a}}{\delta^2}
-\frac{(1-a)K_q}{\delta}+o(1)\right)T^{-2}.
\end{equation*}
Choose an integer $K_q$ with
$(1-a)K_q\delta>C^*_{q,a}+1$.
The coefficient is then strictly negative for all sufficiently large $M$.
Since $L/T\to\delta$, there is $c_q>0$ such that
\begin{equation}\label{eq:critical-local-strict}
\kappa_{q,M}\le\exp(-c_q/L^2).
\end{equation}

The parameters $a$ and $\delta$ are fixed before $K_q$.
Consequently, the positive error constants and the resulting $c_q$
are independent of $M$, $N$ and the local orthonormal basis.

{
For $q=2$, the sharper constants in
\eqref{eq:critical-local-total-loss} allow the explicit choices
\[
a=\frac1{40},\qquad \delta=1,\qquad K_2=1.
\]
Indeed, $L/T\to1$ and
\[
L^2\left(\frac1{L-1}+\log\frac{L-1}{L}\right)\longrightarrow\frac12.
\]
The last three error terms in \eqref{eq:critical-local-total-loss} are
$o(L^{-2})$, \lastrev{uniformly over $r\ge L^4$ and $\mathbf d\in\mathcal C_r(L)$.}
Since $2\gamma_2=1$, we obtain
\begin{equation}
\label{eq:boolean-critical-contraction}
\limsup_{M\to\infty}L^2\log\kappa_{2,M}
\le \frac12+\frac32a^2-(1-a)
=-\frac{1517}{3200}<0.
\end{equation}
Thus \eqref{eq:critical-local-strict} holds with $K_2=1$.
\par}

\subsection{\lastrev{The weighted estimate}}

\lastrev{For the levels $r<L^4$, fix $L_0=3$ and choose $B_{0,q}>b_{q,3}$.}
\rev{For $q=2$, we take $B_{0,2}=5/4$, since
$b_{2,3}=1/(2\log(3/2))<5/4$.}
Proposition~\ref{prop:fixed-block-local} gives 
\begin{equation}  \label{eq:critical-local-baseline}
\left(\sum_{r=1}^M\frac{\mathcal{A}_r(f)^2}{r^{2B_{0,q}}}\right)^{1/2} \le
C(q)M^{3/2}\|f\|_\infty.
\end{equation}
Since $B_{0,q}>2\gamma_q$ and $r<L^4$, 
\begin{equation*}
\frac1{W_q(r)}  \le \frac{r^{B_{0,q}-2\gamma_q}}{r^{B_{0,q}}}  \le \frac{%
L^{4(B_{0,q}-2\gamma_q)}}{r^{B_{0,q}}}. 
\end{equation*}
Hence 
\begin{equation}  \label{eq:critical-local-low}
\left(\sum_{r=1}^{\min\{M,L^4-1\}}\frac{\mathcal{A}_r(f)^2}{W_q(r)^2}\right)^{1/2} \le
C(q)M^{3/2}L^{A_q}\|f\|_\infty, \qquad A_q:=4(B_{0,q}-2\gamma_q).
\end{equation}
Combining \eqref{eq:critical-local-high}, \eqref{eq:critical-local-strict}
and \eqref{eq:critical-local-low}, and taking the supremum in %
\eqref{eq:critical-local-Gamma}, gives 
\begin{equation*}
(\Gamma_{M,N}^{(q)})^2  \le C(q)M^3L^{2A_q} 
+e^{-2c_q/L^2}(\Gamma_{M,N}^{(q)})^2. 
\end{equation*}
Since $1-e^{-2c_q/L^2}\ge c^{\prime }_q/L^2$ for large $L$, 
\begin{equation}  \label{eq:critical-local-Gamma-final}
\Gamma_{M,N}^{(q)} \le C(q)M^{3/2}(1+\log M)^{A_q+1},
\end{equation}
uniformly in $N$ and $\Phi$.

For bounded $M$, the same bound follows from Proposition~\ref%
{prop:fixed-local-levels} by increasing $C(q)$.

\subsection{Completion of the proof}

Set $M=d$ and write $f_r=\sum_{s(\alpha)=r}\widehat f_\Phi(\alpha)\Phi_\alpha
$. Since $p_r\le p_d$ for $r\le d$, 
\begin{equation*}
\|\widehat{f_r}_\Phi\|_{p_d}\le\mathcal{A}_r(f). 
\end{equation*}
H\"older over the degree index and $1/p_d-1/2=1/(2d)$ give 
\begin{align*}
\left(\sum_{1\le s(\alpha)\le d} |\widehat
f_\Phi(\alpha)|^{p_d}\right)^{1/p_d} &\le d^{1/(2d)} \left(\sum_{r=1}^d%
\mathcal{A}_r(f)^2\right)^{1/2} \\
&\le d^{1/(2d)}W_q(d)\Gamma_{d,N}^{(q)}\|f\|_\infty.
\end{align*}
The constant coefficient is bounded by $\|f\|_\infty$. Therefore, after
increasing $C(q)$ and setting $K(q):=K_q+A_q+1$, 
\begin{equation}  \label{eq:unified-BH-final}
\mathsf{B}_d^{\le}(\Omega,\Phi) \le
C(q)d^{\frac32+2\gamma_q}(1+\log d)^{K(q)}.
\end{equation}

\rev{For $q=2$, the choices above give $A_2=4(5/4-1)=1$ and
$K(2)=K_2+A_2+1=3$. This proves the explicit Boolean bound
\eqref{eq:boolean-main-consequence}.}

The \rev{homogeneous} estimate follows by restriction. This proves Theorem~A.

\section{Auxiliary results for the Hardy--Littlewood transfer}
\label{sec:HL-auxiliary}

Fix $q\ge2$, $d\ge2$ and a local orthonormal basis $\Phi$ as in
Theorem~A. For Hardy--Littlewood inequalities we use the
coefficient-dependent rescaling and entropy cancellation of \cite[%
Section~3.2]{PellegrinoTeixeiraHL}, together with a local
square-function estimate \rev{for homogeneous expansions}.

For 
\begin{equation*}
g(x)=\sum_{s(\alpha)=d}a_\alpha\Phi_\alpha(x) 
\end{equation*}
define 
\begin{equation}  \label{eq:S-local-definition}
S_\Phi(g):= \sum_{j=1}^N \left(\sum_{\substack{ s(\alpha)=d \\ j\in\func{supp%
}(\alpha)}}|a_\alpha|^2\right)^{1/2}.
\end{equation}

\attribrev{The next lemma is a local analogue of the derivative
square-function estimate in \cite[Lemma~3.1]{PellegrinoTeixeiraHL}.}

\begin{lemma}
\label{lem:mixed-one-block} For every \rev{homogeneous expansion $g$ of order $d$} on $%
\Omega^N$, 
\begin{equation}  \label{eq:mixed-local-estimate}
S_\Phi(g) \le \pi\sqrt{q-1}\,d(2d-1)\,\rev{e^{2\gamma_q(d-1)}}\|g\|_\infty.
\end{equation}
The constant is independent of $N$ and of the orthonormal basis $\Phi$.
\end{lemma}

\begin{proof}
For $1\le j\le N$ and $1\le u\le q-1$, define
$g_{j,u}:\Omega^N\to\mathbb C$ by 
\begin{equation*}
g_{j,u}(y):=  \sum_{\substack{ s(\alpha)=d \\ \alpha_j=u}} 
a_\alpha\prod_{\substack{1\le k\le N\\k\ne j}}\varphi_{\alpha_k}(y_k). 
\end{equation*}
This function is independent of $y_j$. Parseval gives 
\begin{equation*}
\|g_{j,u}\|_2^2  =\sum_{\substack{ s(\alpha)=d \\ \alpha_j=u}}|a_\alpha|^2, 
\end{equation*}
so 
\begin{equation}  \label{eq:S-by-gju}
S_\Phi(g)\le\sum_{j=1}^N\sum_{u=1}^{q-1}\|g_{j,u}\|_2.
\end{equation}
Each $g_{j,u}$ \rev{has a homogeneous expansion of order $d-1$}.
\rev{Applying \eqref{cyc:reverse-moment} with $M=d-1$ gives}
\begin{equation}  \label{eq:gju-L2-L1}
\|g_{j,u}\|_2\le\rev{e^{2\gamma_q(d-1)}}\|g_{j,u}\|_1.
\end{equation}

Fix $y\in\Omega^N$. Associate with $g$ the $d$-homogeneous polynomial 
\begin{equation*}
P_g:(\mathbb{C}^{q-1})^N\to\mathbb{C},\qquad
P_g(z)=\sum_{s(\alpha)=d}a_\alpha\prod_{j\in\func{supp}\alpha}z_{j,%
\alpha_j}, 
\end{equation*}
which contains at most one variable from each coordinate block. Let $L_g:((%
\mathbb{C}^{q-1})^N)^d\to\mathbb{C}$ be its symmetric $d$-linear form, so $%
P_g(z)=L_g(z,\ldots,z)$. For $x=(x_1,\ldots,x_N)\in\Omega^N$, define the
local evaluation vector 
\begin{equation*}
\zeta_\Phi(x):=(v_{x_1},\ldots,v_{x_N})  \in(\mathbb{C}^{q-1})^N,  \qquad 
v_\xi:=(\varphi_1(\xi),\ldots,\varphi_{q-1}(\xi))\quad(\xi\in\Omega). 
\end{equation*}
For fixed $y\in\Omega^N$, define the mixed evaluation
\begin{equation*}
h_y(x)
:=
L_g\bigl(\zeta_\Phi(x),\zeta_\Phi(y)^{d-1}\bigr),
\qquad
x\in\Omega^N.
\end{equation*}
\introrev{Differentiating $P_g(z)=L_g(z,\ldots,z)$ at
$z=\zeta_\Phi(y)$ gives
\begin{align*}
dh_y(x)
&=DP_g(\zeta_\Phi(y))[\zeta_\Phi(x)]\\
&=\sum_{j=1}^N\sum_{u=1}^{q-1}
\frac{\partial P_g}{\partial z_{j,u}}(\zeta_\Phi(y))\varphi_u(x_j)\\
&=\sum_{j=1}^N\sum_{u=1}^{q-1}g_{j,u}(y)\varphi_u(x_j).
\end{align*}}
Thus $h_y$ \rev{has a homogeneous expansion of order one} and its coefficient
at the local index $(j,u)$ is $g_{j,u}(y)/d$.

\attribrev{We use the random-coordinate mixing argument of
\cite{PellegrinoRaposoCyclic}; the corresponding Boolean
mixed-polarization argument appears in
\cite{DefantMastyloPerez}.}
\lastrev{Fix $x,y\in\Omega^N$ and $0\le t\le1$. On $\Omega^N$ put
\[
\nu_t^{x,y}:=\bigotimes_{k=1}^N
\bigl(t\delta_{x_k}+(1-t)\delta_{y_k}\bigr),
\]
where $\delta_\xi$ is the point mass at $\xi$.
Since each term contains at most one factor from each coordinate,
\begin{align}
\int_{\Omega^N}g(z)\,d\nu_t^{x,y}(z)
&=\sum_{s(\alpha)=d}a_\alpha
\prod_{j\in\func{supp}(\alpha)}
\bigl(t\varphi_{\alpha_j}(x_j)+(1-t)\varphi_{\alpha_j}(y_j)\bigr)\notag\\
&=P_g\bigl(t\zeta_\Phi(x)+(1-t)\zeta_\Phi(y)\bigr)\notag\\
&=\sum_{m=0}^d\binom dm t^m(1-t)^{d-m}
L_g\bigl(\zeta_\Phi(x)^m,\zeta_\Phi(y)^{d-m}\bigr).
\label{eq:Bernstein-mixed-representation}
\end{align}
The triangle inequality gives
\[
\left|\int_{\Omega^N}g\,d\nu_t^{x,y}\right|
\le\int_{\Omega^N}|g|\,d\nu_t^{x,y}
\le\|g\|_\infty,
\qquad
0\le t\le1.
\]}
Moreover, the Bernstein coefficient corresponding to $m=1$ in
\eqref{eq:Bernstein-mixed-representation} is exactly
\[
L_g\bigl(\zeta_\Phi(x),\zeta_\Phi(y)^{d-1}\bigr)
=
h_y(x).
\]

For complex coefficients, multiply the polynomial by a unimodular scalar
so that the coefficient under consideration becomes nonnegative real,
and then take real parts. This operation does not increase the supremum
on $[0,1]$. The coefficient estimate of Lubinsky--Ziegler
\cite{LubinskyZiegler}, also given by Defant, Masty{\l}o and P\'erez
\cite{DefantMastyloPerez} in an equivalent normalization,
therefore gives, at $m=1$,
\[
|h_y(x)|
\le
\frac{\binom{2d}{2}}{\binom d1}\,
\|g\|_\infty
=
(2d-1)\|g\|_\infty.
\]
Since $x$ was arbitrary,
\begin{equation}
\label{eq:mixed-polarization-one}
\|h_y\|_\infty
\le
(2d-1)\|g\|_\infty.
\end{equation}
For a \rev{homogeneous expansion of order one}, $h=\sum_{j=1}^N h_j(x_j)$, Cauchy--Schwarz,
Parseval and Lemma~\ref{lem:phase-sum} give the basis-free estimate 
\begin{equation}  \label{eq:B1-uniform}
\sum_{s(\alpha)=1}|\widehat h_\Phi(\alpha)| \le\pi\sqrt{q-1}\,\|h\|_\infty.
\end{equation}
Applying this to $h_y$, whose coefficients are $g_{j,u}(y)/d$, yields 
\begin{equation*}
\sum_{j=1}^N\sum_{u=1}^{q-1}|g_{j,u}(y)|  \le\pi\sqrt{q-1}\,d(2d-1)\|g\|_\infty. 
\end{equation*}
Integrating in $y$, using \eqref{eq:gju-L2-L1}, and then \eqref{eq:S-by-gju}
proves \eqref{eq:mixed-local-estimate}.
\end{proof}

For $2d\le p\le\infty$ put 
\begin{equation*}
\rho:=\rho_{d,p},\qquad \tau:=\frac{2d}{p}, 
\end{equation*}
with $\tau=0$ at $p=\infty$ and $p_d=2d/(d+1)$ as before. Then
\begin{equation}  \label{eq:transfer-exponent-identities}
\frac1\rho=\frac1{p_d}-\frac1p, \qquad \frac1\rho-\frac12=\frac1{2d}-\frac1p.
\end{equation}

\Needspace{9\baselineskip}

\begin{lemma}
\label{lem:local-rescaling} Let $2d\le p<\infty$ and 
\begin{equation*}
f(x)=\sum_{s(\alpha)=d}a_\alpha\Phi_\alpha(x)\ne0. 
\end{equation*}
Set 
\begin{equation*}
A:=\left(\sum_{s(\alpha)=d}|a_\alpha|^\rho\right)^{1/\rho},  \qquad  \mu_\alpha:=%
\frac{|a_\alpha|^\rho}{A^\rho}, 
\end{equation*}
and 
\begin{equation*}
t_j:=\frac1d\sum_{\substack{ s(\alpha)=d \\ j\in\func{supp}(\alpha)}}%
\mu_\alpha. 
\end{equation*}
Then $t=(t_j)$ is a probability vector. Define 
\begin{equation*}
Q(x):=\sum_{s(\alpha)=d}a_\alpha
\left(\prod_{j\in\func{supp}(\alpha)}t_j^{1/p}\right)\Phi_\alpha(x), 
\end{equation*}
omitting coordinates with $t_j=0$. Then 
\begin{equation}  \label{eq:Q-local-norm}
\|Q\|_\infty\le\|f\|_{p,\Phi}
\end{equation}
and 
\begin{equation}  \label{eq:local-rescaling-conclusion}
A\le \|\widehat Q_\Phi\|_{\ell^{p_d}}^{1-\tau} \left(\frac{S_\Phi(Q)}{\sqrt d}\right)^\tau.
\end{equation}
\end{lemma}

\begin{proof}
Since \rev{every index with a nonzero coefficient has exactly $d$ active coordinates}, $%
\sum_jt_j=1$. Hence $(t_j^{1/p})_j$ has $\ell_p$ norm one and %
\eqref{eq:Q-local-norm} follows from the definition of $\|\cdot\|_{p,\Phi}$.
If $a_\alpha\ne0$, then $t_j>0$ for every $j\in\func{supp}(\alpha)$,
so $Q\ne0$. \introrev{Throughout this proof, $\sum_\alpha$ is over
$\alpha\in\{0,\ldots,q-1\}^N$ with $s(\alpha)=d$, and $\sum_j$ is over
$1\le j\le N$. In $\sum_{j,\alpha}$ we also require
$j\in\operatorname{supp}(\alpha)$. Logarithmic terms are evaluated on
positive masses, with $0\log0=0$ in the entropy expressions.}

Put 
\begin{equation*}
H(t):=-\sum_jt_j\log t_j,  \qquad 
H(\mu):=-\sum_\alpha\mu_\alpha\log\mu_\alpha,  \qquad  I:=dH(t)-H(\mu). 
\end{equation*}
The entropy calculation is the one used in \cite[Lemma~3.2]%
{PellegrinoTeixeiraHL}; \rev{here each radial variable has exponent one
in every term in which it occurs}. Since $p_d/\rho=1-p_d/p$, 
\begin{align*}
\left(\frac{\|\widehat Q_\Phi\|_{\ell^{p_d}}}{A}\right)^{p_d} &=\sum_\alpha
\mu_\alpha^{p_d/\rho} \prod_{j\in\func{supp}(\alpha)}t_j^{p_d/p} \\
&=\sum_\alpha\mu_\alpha \exp\left\{\frac{p_d}{p} \left(\sum_{j\in\func{supp}%
(\alpha)}\log t_j-\log\mu_\alpha\right)\right\}.
\end{align*}
\lastrev{The definition of $t_j$ gives}
\begin{equation*}
\sum_\alpha\mu_\alpha  \sum_{j\in\func{supp}(\alpha)}\log t_j 
=d\sum_jt_j\log t_j=-dH(t),  \qquad 
\sum_\alpha\mu_\alpha\log\mu_\alpha=-H(\mu). 
\end{equation*}
\lastrev{Jensen's inequality for the exponential therefore gives
\[
\left(\frac{\|\widehat Q_\Phi\|_{\ell^{p_d}}}{A}\right)^{p_d}
\ge\exp\left(-\frac{p_d I}{p}\right),
\]
and hence}
\begin{equation}  \label{eq:Q-r0-lower}
\|\widehat Q_\Phi\|_{\ell^{p_d}}\ge Ae^{-I/p}.
\end{equation}

Next set 
\begin{equation*}
b_\alpha:=a_\alpha\prod_{i\in\func{supp}(\alpha)}t_i^{1/p},  \qquad 
S_j:=\left(\sum_{\alpha:\,j\in\func{supp}(\alpha)}|b_\alpha|^2\right)^{1/2}. 
\end{equation*}
Thus $S_\Phi(Q)=\sum_jS_j$. On pairs $(j,\alpha)$ with $j\in\func{supp}%
(\alpha)$ put 
\begin{equation*}
\lastrev{w_{j,\alpha}:=\frac{\mu_\alpha}{d}},  \qquad  \beta_{\alpha\mid j}:=\frac{%
\mu_\alpha}{dt_j}\quad(t_j>0). 
\end{equation*}
\lastrev{These weights satisfy
\[
\sum_{\alpha:\,j\in\func{supp}(\alpha)}w_{j,\alpha}=t_j,
\qquad
\sum_{j\in\func{supp}(\alpha)}w_{j,\alpha}=\mu_\alpha,
\qquad
\sum_{\alpha:\,j\in\func{supp}(\alpha)}\beta_{\alpha\mid j}=1
\quad(t_j>0).
\]}
Weighted
Jensen applied first to $\sum_jS_j/A$ with weights $t_j$, and then to $%
S_j^2/A^2$ with weights $\beta_{\alpha\mid j}$, gives 
\lastrev{\begin{align*}
\log\frac{S_\Phi(Q)}A
&=\log\sum_{j:t_j>0}t_j\frac{S_j}{At_j}
\ge H(t)+\sum_jt_j\log\frac{S_j}{A} \\
&\ge H(t)+\frac12 \sum_{j,\alpha}w_{j,\alpha} \log\left( \frac{%
\mu_\alpha^{2/\rho} \prod_{i\in\func{supp}(\alpha)}t_i^{2/p}} {%
\beta_{\alpha\mid j}} \right).
\end{align*}}
\lastrev{Using these identities and $\beta_{\alpha\mid j}=\mu_\alpha/(dt_j)$, the}
last expression equals 
\begin{align*}
\frac12\log d +\left(\frac12-\frac dp\right)H(t)
-\left(\frac1\rho-\frac12\right)H(\mu) &=\frac12\log d
+\left(\frac1{2d}-\frac1p\right)I,
\end{align*}
where \eqref{eq:transfer-exponent-identities} is used in the last step.
Hence 
\begin{equation}  \label{eq:S-entropy-lower}
\frac{S_\Phi(Q)}{\sqrt d} \ge
A\exp\left\{\left(\frac1{2d}-\frac1p\right)I\right\}.
\end{equation}
Raising \eqref{eq:Q-r0-lower} to $1-\tau$ and \eqref{eq:S-entropy-lower} to $%
\tau$, the entropy factors cancel because 
\begin{equation*}
-\frac{1-\tau}{p}+\tau\left(\frac1{2d}-\frac1p\right)=0. 
\end{equation*}
This proves \eqref{eq:local-rescaling-conclusion}.
\end{proof}

\section{Theorem B and its consequences}
\label{sec:proof-B}

\subsection{Proof of Theorem B}

\begin{proof}
Assume first $2d\le p<\infty$ and let $f\ne0$ \rev{have a homogeneous
expansion of order $d$}; the zero function is immediate. Apply
Lemma~\ref{lem:local-rescaling}. Since $Q$ \rev{has the same property}, 
\begin{equation*}
\|\widehat Q_\Phi\|_{\ell^{p_d}}  \le\mathsf{B}_d^{=} (\Omega,\Phi)\|Q\|_\infty. 
\end{equation*}
Lemma~\ref{lem:mixed-one-block} gives 
\begin{equation*}
\frac{S_\Phi(Q)}{\sqrt d}  \le \Lambda_{d,q}\|Q\|_\infty. 
\end{equation*}
Together with \eqref{eq:Q-local-norm}, this proves 
\begin{equation*}
\mathsf{H}_{d,p}(\Phi)  \le  \mathsf{B}_d^{=} (\Omega,\Phi)^{1-\tau} 
\Lambda_{d,q}^{\tau}. 
\end{equation*}
At $p=\infty$, $\tau=0$ and the estimate follows from $\|f\|_\infty\le\|f%
\|_{\infty,\Phi}$ and the definition of $\mathsf{B}_d^{=} (\Omega,\Phi)$.
\end{proof}

\subsection{Growth regimes}

\begin{lemma}
\label{lem:local-lower} For every uniform local orthonormal basis, $\mathsf{B%
}_d^{=}(\Omega,\Phi)\ge1$. For $p\ge2d$, $\mathsf{H}_{d,p}(\Phi)\ge d^{d/p}$.
\end{lemma}

\begin{proof}
Label the points of $\Omega$ by $0,\ldots,q-1$ and put $h(j)=e^{2\pi ij/q}$.
This function is mean zero, unimodular, and has $L_2$ norm one. Write $%
h=\sum_{u=1}^{q-1}c_u\varphi_u$, so $\sum_u|c_u|^2=1$. The function $%
f(x)=\prod_{j=1}^dh(x_j)$ \lastrev{on $\Omega^d$} \rev{has a homogeneous expansion of order $d$},
has supremum norm one,
and its coefficient norm at any exponent $1\le t\le2$ is $\|c\|_t^d\ge1$.
For finite $p$, the arithmetic--geometric mean inequality applied to
$|r_j|^p$ gives
\lastrev{\[
\prod_{j=1}^d|r_j|
\le\left(\frac1d\sum_{j=1}^d|r_j|^p\right)^{d/p}
\le d^{-d/p},
\]
with equality when $r_j=d^{-1/p}$ for all $j$. Thus
$\|f\|_{p,\Phi}=d^{-d/p}$.} At $p=\infty$ this norm is
one, giving the same formula with $1/\infty=0$. Use the coefficient
exponents $p_d$ and $\rho_{d,p}$, respectively.
\end{proof}

\begin{proof}[Proof of Corollary~\protect\ref{cor:HL-regimes}]
Combining Theorems~A and B gives
\begin{equation}
\label{eq:HL-master}
\log\mathsf{H}_{d,p}(\Phi)
\le
\beta_q\log d
+
K(q)\log(1+\log d)
+
\rev{4\gamma_q}\frac{d^2}{p}
+
O_q\!\left(1+\frac{d\log d}{p}\right).
\end{equation}
Indeed,
\[
\log\Lambda_{d,q}
=
\rev{2\gamma_q}d
+
\frac32\log d
+
O_q(1),
\]
\lastrev{and Lemma~\ref{lem:local-lower} gives $\mathsf B_d^{=}\ge1$, so that
$(1-\tau)\log\mathsf B_d^{=}\le\log\mathsf B_d^{=}$.}

Assume first that
\[
\frac{p(d)}d\longrightarrow\infty.
\]
Then
\[
\frac{d^2}{p(d)}=o(d)
\qquad\text{and}\qquad
\frac{d\log d}{p(d)}=o(d),
\]
while $\log d=o(d)$. Hence \eqref{eq:HL-master} gives
\[
\log\mathsf{H}_{d,p(d)}(\Phi)\le o(d).
\]
On the other hand, Lemma~\ref{lem:local-lower} gives
\[
\mathsf{H}_{d,p(d)}(\Phi)\ge1,
\]
and therefore
\[
0
\le
\log\mathsf{H}_{d,p(d)}(\Phi)
=
o(d).
\]
Equivalently,
\[
\mathsf{H}_{d,p(d)}(\Phi)=\exp(o(d)).
\]

Next suppose that, for some fixed $c>0$,
\[
p(d)\ge c\,\frac{d^2}{\log d}.
\]
Then
\[
\frac{d^2}{p(d)}
\le
\frac1c\log d,
\qquad
\frac{d\log d}{p(d)}
=
O_c\!\left(\frac{(\log d)^2}{d}\right)
=
o(1).
\]
Substitution in \eqref{eq:HL-master} yields
\[
\mathsf{H}_{d,p(d)}(\Phi)
\le
C(q,c)\,
d^{\,\beta_q+\rev{4\gamma_q/c}}
(1+\log d)^{K(q)},
\]
after increasing $C(q,c)$ if necessary.

Finally, assume that
\[
\frac{p(d)}{d^2}\longrightarrow\infty.
\]
Theorem~B and the lower bound
$\mathsf{B}_d^{=}(\Omega,\Phi)\ge1$ imply
\[
\frac{\mathsf{H}_{d,p(d)}(\Phi)}
{\mathsf{B}_d^{=}(\Omega,\Phi)}
\le
\Lambda_{d,q}^{\,2d/p(d)}.
\]
Since
\[
\frac{2d}{p(d)}\log\Lambda_{d,q}
=
O_q\!\left(\frac{d^2}{p(d)}
+
\frac{d\log d}{p(d)}\right)
=o(1),
\]
the right-hand side is $1+o(1)$. Consequently
\[
\mathsf{H}_{d,p(d)}(\Phi)
\le
(1+o(1))
\mathsf{B}_d^{=}(\Omega,\Phi),
\]
uniformly over the local orthonormal basis $\Phi$.
\end{proof}

\begin{corollary}
\label{cor:further-regimes} Uniformly over $\Phi$, if $p(d)/(d^2/\log
d)\to\infty$, then $\mathsf{H}_{d,p(d)}(\Phi)\le d^{\beta_q+o(1)}$. If $%
p(d)\ge cd^2$, then 
\begin{equation*}
\mathsf{H}_{d,p(d)}(\Phi)\le C(q,c)d^{\beta_q}(1+\log d)^{K(q)}. 
\end{equation*}
If $p(d)/d^2\to c\in(0,\infty]$, then 
\begin{equation*}
\limsup_{d\to\infty}\frac{\mathsf{H}_{d,p(d)}(\Phi)}{\mathsf{B}%
_d^{=}(\Omega,\Phi)}\le \rev{e^{4\gamma_q/c}}, 
\end{equation*}
with $1/\infty=0$.
\end{corollary}

\begin{proof}
The first two assertions follow by inserting the respective assumptions into %
\eqref{eq:HL-master}. For the last, take logarithms of the ratio estimate in
the preceding proof and use the expansion of $\log\Lambda_{d,q}$.
\end{proof}

\section{Finite local systems: geometric preliminaries}
\label{sec:local-systems}

Let $\Omega$ be a nonempty finite set.

\begin{definition}
\label{def:local-system} A \emph{finite local system} on $\Omega$ is a
linearly independent family 
\begin{equation*}
\Phi=(\phi_0,\phi_1,\ldots,\phi_\nu),  \qquad  \phi_u:\Omega\longrightarrow%
\mathbb{C},  \qquad  \phi_0\equiv1, 
\end{equation*}
where $\nu\ge1$.
\end{definition}

Fix $N\in\mathbb{N}$. For 
\begin{equation*}
\alpha=(\alpha_1,\ldots,\alpha_N)\in\{0,1,\ldots,\nu\}^N 
\end{equation*}
we use the previously fixed notation $\func{supp}(\alpha)$ and $s(\alpha)$,
and define the corresponding product function 
\begin{equation}  \label{eq:product-function}
\Phi_\alpha:\Omega^N\longrightarrow\mathbb{C}, \qquad \Phi_\alpha(x)
:=\prod_{j\in\func{supp}(\alpha)}\phi_{\alpha_j}(x_j), \qquad
x=(x_1,\ldots,x_N).
\end{equation}
The empty product is understood to be $1$.

\begin{lemma}
\label{lem:product-independence} The family 
\begin{equation*}
\bigl(\Phi_\alpha\bigr)_{\alpha\in\{0,\ldots,\nu\}^N} 
\end{equation*}
is linearly independent in the space of complex-valued functions on $\Omega^N
$.
\end{lemma}

\begin{proof}
Suppose 
\begin{equation*}
\sum_{\alpha\in\{0,\ldots,\nu\}^N}c_\alpha\Phi_\alpha(x)=0  \qquad
(x\in\Omega^N). 
\end{equation*}
Group the sum according to the first coordinate of $\alpha$: 
\begin{equation*}
\sum_{u=0}^\nu  \phi_u(x_1)F_u(x_2,\ldots,x_N)=0, 
\end{equation*}
where each $F_u$ is a linear combination of product functions in the
remaining $N-1$ variables. Since $\phi_0,\ldots,\phi_\nu$ are linearly
independent on $\Omega$, every $F_u$ vanishes identically. Applying this
successively to $x_2,\ldots,x_N$ gives $c_\alpha=0$ for every $\alpha$.
\end{proof}

Define 
\begin{equation*}
\mathcal{W}(\Phi,N)  :=\func{span} 
\{\Phi_\alpha:\alpha\in\{0,\ldots,\nu\}^N\}. 
\end{equation*}
By Lemma~\ref{lem:product-independence}, every $f\in\mathcal{W}(\Phi,N)$ has
a unique expansion 
\begin{equation}  \label{eq:expansion}
f=\sum_{\alpha\in\{0,\ldots,\nu\}^N} \widehat f_\Phi(\alpha)\Phi_\alpha.
\end{equation}
\introrev{The coefficients $\widehat f_\Phi(\alpha)$ are defined by
\eqref{eq:expansion}; no orthogonality of the local family is assumed.}

For integers $1\le d\le N$, define \rev{the space of homogeneous
expansions of order $d$ by} 
\begin{equation}  \label{eq:Vd}
\mathcal{V}_d(\Phi,N) :=\left\{ f\in\mathcal{W}(\Phi,N): \widehat
f_\Phi(\alpha)=0 \text{ whenever }s(\alpha)\ne d \right\}.
\end{equation}

\begin{definition}
\label{def:lpgeometry} Let $1\le p\le\infty$ and $f\in\mathcal{V}_d(\Phi,N)$%
. Define 
\begin{equation}  \label{eq:extension}
\widetilde f(\mathbf r,x) := \sum_{s(\alpha)=d} \widehat f_\Phi(\alpha) \prod_{j\in%
\func{supp}(\alpha)}r_j\phi_{\alpha_j}(x_j), \qquad \mathbf r\in\mathbb{C}^N, \quad
x\in\Omega^N,
\end{equation}
and 
\begin{equation}  \label{eq:norm}
\left\| f\right\|_{p,\Phi} := \sup_{\left\| \mathbf r\right\|_{\ell_p^N}\le1}
\sup_{x\in\Omega^N} \left|\widetilde f(\mathbf r,x)\right|.
\end{equation}
\end{definition}

\introrev{The quantity in \eqref{eq:norm} is a norm on $\mathcal V_d(\Phi,N)$.
To verify positive definiteness for $1\le p<\infty$, take} $%
\mathbf r=N^{-1/p}(1,\ldots,1)$. Then $\left\| \mathbf r\right\|_p=1$ and, because every
term in \eqref{eq:extension} contains exactly $d$ radial factors, 
\begin{equation*}
\widetilde f(\mathbf r,x)=N^{-d/p}f(x). 
\end{equation*}
Thus $f\ne0$ implies $\left\| f\right\|_{p,\Phi}>0$. For $p=\infty$, take $%
\mathbf r=(1,\ldots,1)$.

\introrev{For each $x\in\Omega^N$, the polynomial
$\widetilde f(\cdot,x)$ is $d$-homogeneous, since every term contains
one radial factor for each active coordinate.}

Associate with $f\in\mathcal{V}_d(\Phi,N)$ the block-square-free polynomial 
\begin{equation}  \label{eq:Pf}
P_f(z) := \sum_{s(\alpha)=d} \widehat f_\Phi(\alpha) \prod_{j\in\func{supp}%
(\alpha)}z_{j,\alpha_j}, \qquad z=(z_{j,u})_{\substack{ 1\le j\le N \\ 1\le
u\le\nu}} \in\mathbb{C}^{N\nu}.
\end{equation}
It is $d$-homogeneous and contains at most one coordinate from each block $%
\{z_{j,1},\ldots,z_{j,\nu}\}$.

For $x\in\Omega$, define the local value vector 
\begin{equation*}
v_x:=(\phi_1(x),\ldots,\phi_\nu(x))\in\mathbb{C}^\nu. 
\end{equation*}
For a subset $A$ of a complex vector space, write 
\begin{equation*}
\func{absconv}(A)  :=\left\{  \sum_{k=1}^m\lambda_k a_k:  m\in\mathbb{N},\
\lambda_k\in\mathbb C,\ a_k\in A\ (1\le k\le m),\
\sum_{k=1}^m|\lambda_k|\le1  \right\}. 
\end{equation*}
Set 
\begin{equation}  \label{eq:Ephi}
B_{E_\Phi} :=\func{absconv}\{v_x:x\in\Omega\}.
\end{equation}
The vectors $(v_x)_{x\in\Omega}$ span $\mathbb{C}^\nu$. Indeed, if $%
c=(c_1,\ldots,c_\nu)$ satisfies 
\begin{equation*}
\sum_{u=1}^\nu c_u\phi_u(x)=0  \qquad(x\in\Omega), 
\end{equation*}
then the linear independence of $\phi_0,\phi_1,\ldots,\phi_\nu$ implies $%
c_1=\cdots=c_\nu=0$. Since $\Omega$ is finite, the absolute convex hull in
\eqref{eq:Ephi} is compact. It is also balanced, convex and
absorbing, with the origin in its interior, and hence is the
closed unit ball of a norm on $\mathbb{C}^\nu$. We denote this norm by $\left\|
\cdot\right\|_{E_\Phi}$.

For $1\le p\le\infty$ set 
\begin{equation*}
X_{p,\Phi}^N:=\ell_p^N(E_\Phi). 
\end{equation*}
For $p<\infty$, 
\begin{equation*}
\left\| z\right\|_{X_{p,\Phi}^N}  :=\left(\sum_{j=1}^N\left\|
z_j\right\|_{E_\Phi}^p\right)^{1/p}, 
\end{equation*}
and for $p=\infty$, 
\begin{equation*}
\left\| z\right\|_{X_{\infty,\Phi}^N}  :=\max_{1\le j\le N}\left\|
z_j\right\|_{E_\Phi}. 
\end{equation*}
A useful decomposition of its unit ball is 
\begin{equation}  \label{eq:block-ball}
B_{X_{p,\Phi}^N} = \bigcup_{\substack{ a_j\ge0 \\ \left\|
a\right\|_{\ell_p^N}\le1}} \prod_{j=1}^N a_jB_{E_\Phi}.
\end{equation}
Indeed, for $z=(z_1,\ldots,z_N)$ one may take $a_j=\left\|
z_j\right\|_{E_\Phi}$.

\begin{proposition}
\label{prop:geometry} For every $f\in\mathcal{V}_d(\Phi,N)$ and every $1\le
p\le\infty$, 
\begin{equation}  \label{eq:geometry}
\sup_{\left\| z\right\|_{X_{p,\Phi}^N}\le1}|P_f(z)| =\left\|
f\right\|_{p,\Phi}.
\end{equation}
\end{proposition}

\begin{proof}
Fix nonnegative $a_1,\ldots,a_N$ with $\left\| a\right\|_p\le1$. By %
\eqref{eq:block-ball}, it is enough to maximize $|P_f|$ on $%
\prod_{j=1}^Na_jB_{E_\Phi}$.

Fix all blocks except the $j$th. Since $P_f$ is block-square-free, there are
a scalar $A$ and a linear functional $L$ on $E_\Phi$ such that 
\begin{equation*}
P_f(z)=A+L(z_j). 
\end{equation*}
Moreover, 
\begin{equation*}
a_jB_{E_\Phi}  =\func{conv}  \{\lambda v_x:|\lambda|\le a_j,\ x\in\Omega\}. 
\end{equation*}
If 
\begin{equation*}
z_j=\sum_{k=1}^m\theta_k\lambda_k v_{x_k},  \qquad
\theta_k\ge0,\quad |\lambda_k|\le a_j,\quad x_k\in\Omega
\quad(1\le k\le m),\qquad \sum_{k=1}^m\theta_k=1, 
\end{equation*}
then 
\begin{equation*}
A+L(z_j)  =\sum_{k=1}^m  \theta_k\bigl(A+L(\lambda_k v_{x_k})\bigr), 
\end{equation*}
and hence 
\begin{equation*}
|A+L(z_j)|  \le  \sup_{\substack{ |\lambda|\le a_j \\ x\in\Omega}} 
|A+L(\lambda v_x)|. 
\end{equation*}
The reverse inequality is immediate because each $\lambda v_x$ belongs to $%
a_jB_{E_\Phi}$. Therefore 
\begin{equation*}
\sup_{z_j\in a_jB_{E_\Phi}}|A+L(z_j)|  =  \sup_{\substack{ |\lambda|\le a_j
\\ x\in\Omega}}  |A+L(\lambda v_x)|. 
\end{equation*}
Applying this reduction successively to $j=1,\ldots,N$ gives 
\begin{align*}
\sup_{z\in\prod_{j=1}^Na_jB_{E_\Phi}}|P_f(z)| &\ = \sup_{\substack{\mathbf r\in\mathbb C^N,\ x\in\Omega^N\\
|r_j|\le a_j\ (1\le j\le N)}} \left| \sum_{s(\alpha)=d} \widehat f_\Phi(\alpha) \prod_{j\in%
\func{supp}(\alpha)}r_j\phi_{\alpha_j}(x_j) \right|.
\end{align*}
Finally, taking the supremum over all nonnegative $a\in B_{\ell_p^N}$ is equivalent to
requiring $\left\| \mathbf r\right\|_p\le1$. This is exactly \eqref{eq:norm}.
\end{proof}

\medskip The local geometry in Theorem~C is measured by the following constant.

\begin{definition}
\label{def:kappa} For $1\le p\le\infty$, define 
\begin{equation}  \label{eq:kappa}
\kappa_{p,\Phi} :=\left\|\func{id}:\ell_p^\nu\longrightarrow E_\Phi\right\|
=\sup_{y\in\mathbb{C}^\nu\setminus\{0\}} \frac{\|y\|_{E_\Phi}}{%
\|y\|_{\ell_p^\nu}}.
\end{equation}
\end{definition}

\begin{lemma}
\label{lem:comparison} For $1\le p\le\infty$ and every $f\in\mathcal{V}_d(\Phi,N)$, 
\begin{equation}  \label{eq:comparison}
\sup_{\|z\|_{\ell_p^{N\nu}}\le1}|P_f(z)| \le
\kappa_{p,\Phi}^{\,d}\|f\|_{p,\Phi}.
\end{equation}
\end{lemma}

\begin{proof}
Write $z=(z_1,\ldots,z_N)$ with $z_j\in\mathbb{C}^\nu$. By \eqref{eq:kappa}, 
$\|z_j\|_{E_\Phi}\le\kappa_{p,\Phi}\|z_j\|_{\ell_p^\nu}$, hence 
\begin{equation*}
\|z\|_{X_{p,\Phi}^N}\le\kappa_{p,\Phi}\|z\|_{\ell_p^{N\nu}}. 
\end{equation*}
If $\|z\|_{\ell_p^{N\nu}}\le1$, then $\kappa_{p,\Phi}^{-1}z$ belongs to the
unit ball of $X_{p,\Phi}^N$. Since $P_f$ is $d$-homogeneous, Proposition~\ref%
{prop:geometry} gives 
\begin{equation*}
|P_f(z)|\le\kappa_{p,\Phi}^{d}\|f\|_{p,\Phi}. 
\end{equation*}
Taking the supremum proves the claim.
\end{proof}

\section{Proof of Theorem C: coefficient inequality and optimality}
\label{sec:proof-C}

\subsection{The coefficient inequality}

\label{sec:HL-estimate} For $d\ge1$ and $d<p\le\infty$, let $\mathsf{H}^{\mathrm{pol}%
}_{d,p}$ denote the optimal dimension-free constant in the classical complex
polynomial Hardy--Littlewood inequality 
\begin{equation*}
\left(\sum_{|\beta|=d}|c_\beta|^{\rho_{d,p}}\right)^{1/\rho_{d,p}}  \le 
\mathsf{H}^{\mathrm{pol}}_{d,p}  \sup_{\|z\|_{\ell_p^M}\le1}|P(z)| 
\end{equation*}
for every $M\in\mathbb N$ and every $d$-homogeneous polynomial
$P(z)=\sum_{|\beta|=d}c_\beta z^\beta$ on $\mathbb C^M$.
Here $\beta\in\mathbb N_0^M$, $|\beta|=\sum_{j=1}^M\beta_j$ and
$z^\beta=\prod_{j=1}^M z_j^{\beta_j}$; see \cite%
{AlbuquerqueBayartPellegrinoSeoane,DimantSevillaPeris}.

The optimal local constant is 
\begin{equation}  \label{eq:Hphi}
\mathsf{H}_{d,p}(\Phi) :=\sup_{N\ge d}\sup_{0\ne f\in\mathcal{V}_d(\Phi,N)} 
\frac{\left(\sum_{s(\alpha)=d}|\widehat
f_\Phi(\alpha)|^{\rho_{d,p}}\right)^{1/\rho_{d,p}}}{\|f\|_{p,\Phi}}.
\end{equation}

\begin{proof}[Proof of the upper estimate in Theorem C]
Let $f\in\mathcal{V}_d(\Phi,N)$ and consider the block-square-free
polynomial $P_f$ from \eqref{eq:Pf}. For each interaction index $\alpha$
with $s(\alpha)=d$, the corresponding monomial in $P_f$ uses exactly one
variable from each active block. \introrev{Its coefficient is $\widehat f_\Phi(\alpha)$, and distinct indices
give distinct monomials.}

Applying the classical polynomial Hardy--Littlewood inequality to $P_f$ and
then Proposition~\ref{prop:geometry} together with Lemma~\ref{lem:comparison}
gives 
\begin{align*}
\left(\sum_{s(\alpha)=d}|\widehat
f_\Phi(\alpha)|^{\rho_{d,p}}\right)^{1/\rho_{d,p}} &\le \mathsf{H}^{\mathrm{%
pol}}_{d,p} \sup_{\|z\|_{\ell_p^{N\nu}}\le1}|P_f(z)| \\
&\le \mathsf{H}^{\mathrm{pol}}_{d,p}\,\kappa_{p,\Phi}^{d}\|f\|_{p,\Phi}.
\end{align*}
Taking the optimal constant proves the upper estimate in Theorem C. The
sharpness assertion is proved in Section~\ref{sec:HL-optimality}.
\end{proof}

\subsection{Sharpness of the Hardy--Littlewood exponent}

\label{sec:HL-optimality}

Fix any $u_0\in\{1,\ldots,\nu\}$ and put
$M_{u_0}:=\|\phi_{u_0}\|_\infty>0$. Linear independence ensures that
$\phi_{u_0}$ is nonconstant and nonzero. The following identity allows
sharpness to be tested in this single local direction, without assuming
unimodularity.

\begin{lemma}
\label{lem:squarefree} Let $N\ge d$ and $1\le p\le\infty$. For coefficients $(a_A)_{|A|=d}$ define
\begin{equation}\label{eq:squarefree-f}
f(x):=\sum_{\substack{A\subseteq[N]\\|A|=d}}
a_A\prod_{j\in A}\phi_{u_0}(x_j).
\end{equation}
For each $A$, let $\alpha^A_j=u_0$ if $j\in A$ and $\alpha^A_j=0$
otherwise. Then $\widehat f_\Phi(\alpha^A)=a_A$, all other coefficients
vanish, and
\begin{equation}\label{eq:squarefreeidentity}
\|f\|_{p,\Phi}=M_{u_0}^{d}
\sup_{\|z\|_{\ell_p^N}\le1}\left|\sum_{|A|=d}a_Az_A\right|,
\qquad z_A:=\prod_{j\in A}z_j.
\end{equation}
\end{lemma}

\begin{proof}
The coefficient assertion follows from product independence.
Write $P(z)=\sum_{|A|=d}a_Az_A$. For an admissible pair $(\mathbf r,x)$, put
$z_j=r_j\phi_{u_0}(x_j)/M_{u_0}$. Then $|z_j|\le|r_j|$, so
$\|z\|_p\le1$, and $\widetilde f(\mathbf r,x)=M_{u_0}^dP(z)$. This proves
one inequality in \eqref{eq:squarefreeidentity}.

Conversely, choose $x^0\in\Omega$ with
$|\phi_{u_0}(x^0)|=M_{u_0}$. Given $\|z\|_p\le1$, set
$r_j=M_{u_0}z_j/\phi_{u_0}(x^0)$ and $x=(x^0,\ldots,x^0)$.
Then $\|\mathbf r\|_p=\|z\|_p\le1$ and
$\widetilde f(\mathbf r,x)=M_{u_0}^dP(z)$. Taking suprema proves the reverse
inequality, including at $p=\infty$.
\end{proof}

\attribrev{The next lemma uses the square-free polynomial form of the
diagonal multilinear test in
\cite[proof of Proposition~4.1]{DimantSevillaPeris}.}

\begin{lemma}
\label{lem:disjoint} Fix $u_0$ and $M_{u_0}$ as above, let $d<p<\infty$, and let $M\in\mathbb{N}$. Put $N:=Md$ and
partition $\{1,\ldots,N\}$ into the disjoint sets 
\begin{equation*}
A_k:=\{(k-1)d+1,\ldots,kd\},  \qquad 1\le k\le M. 
\end{equation*}
For 
\begin{equation*}
f_M(x):=\sum_{k=1}^M\prod_{j\in A_k}\phi_{u_0}(x_j) 
\end{equation*}
one has 
\begin{equation}  \label{eq:disjoint-norm}
\left\| f_M\right\|_{p,\Phi} =M_{u_0}^d d^{-d/p}M^{1-d/p}.
\end{equation}
\end{lemma}

\begin{proof}
By Lemma~\ref{lem:squarefree}, the norm in \eqref{eq:disjoint-norm} equals 
\begin{equation*}
M_{u_0}^d\sup_{\left\| z\right\|_{\ell_p^{Md}}\le1}  \left|\sum_{k=1}^M
z_{A_k}\right|. 
\end{equation*}
For such a vector $z$, set 
\begin{equation*}
b_k:=\left(\sum_{j\in A_k}|z_j|^p\right)^{1/p}. 
\end{equation*}
The arithmetic--geometric mean inequality in each block gives 
\begin{equation*}
|z_{A_k}|\le d^{-d/p}b_k^d. 
\end{equation*}
Moreover $\sum_{k=1}^M b_k^p\le1$. Since $d<p$, H\"older's inequality yields 
\begin{equation*}
\sum_{k=1}^M b_k^d  \le M^{1-d/p}\left(\sum_{k=1}^M b_k^p\right)^{d/p}  \le
M^{1-d/p}. 
\end{equation*}
Hence the left-hand side of \eqref{eq:disjoint-norm} is at most
$M_{u_0}^d d^{-d/p}M^{1-d/p}$. Equality is attained by taking $|z_j|=(Md)^{-1/p}$ for
every $j$ and choosing all phases equal. This proves \eqref{eq:disjoint-norm}%
.
\end{proof}

\begin{lemma}
\label{lem:KSZ} Fix $d\ge1$ and $2\le p\le\infty$, with $1/\infty:=0$. There is a constant $K_d>0$,
independent of $N$ and $p$, such that for every $N\ge d$ there are signs
$\varepsilon_A\in\{-1,1\}$, $|A|=d$, satisfying 
\begin{equation}  \label{eq:KSZ}
\sup_{\left\| z\right\|_{\ell_p^N}\le1} \left| \sum_{|A|=d}\varepsilon_Az_A
\right| \le K_dN^{\frac{d+1}{2}-\frac dp}.
\end{equation}
\end{lemma}

\begin{proof}
\lastrev{Let $S_{N,d}:=\{-1,1\}^{\{A\subseteq[N]:|A|=d\}}$ with its uniform
probability measure $\tau_{N,d}$. The coordinate functions
$\varepsilon\mapsto\varepsilon_A$ are independent Rademacher variables.
Put
\[
Z(\varepsilon):=\sup_{\|z\|_{\ell_p^N}\le1}
\left|\sum_{\substack{A\subseteq[N]\\|A|=d}}\varepsilon_Az_A\right|.
\]
Corollary~6.15 of \cite{DefantMastylo}, specialized to these
square-free multi-indices, gives}
\begin{equation*}
\lastrev{\|Z\|_{L_{\psi_2}(\tau_{N,d})} \le C}\sqrt{%
N(1+\log d)}  \sup_{\left\| z\right\|_p\le1} 
\left(\sum_{|A|=d}|z_A|^2\right)^{1/2}, 
\end{equation*}
where $C>0$ is absolute. Here $\psi_2(t)=e^{t^2}-1$, and
$L_{\psi_2}$ denotes the associated Orlicz space with its Luxemburg norm.
\lastrev{For every $\lambda>0$ with
$\int_{S_{N,d}}e^{Z^2/\lambda^2}\,d\tau_{N,d}\le2$,
Cauchy--Schwarz and Jensen's inequality for the exponential give
\[
\int_{S_{N,d}}Z\,d\tau_{N,d}
\le\left(\int_{S_{N,d}}Z^2\,d\tau_{N,d}\right)^{1/2}
\le\lambda\sqrt{\log2}.
\]
Taking the infimum over such $\lambda$ gives
$\int_{S_{N,d}}Z\,d\tau_{N,d}
\le\sqrt{\log2}\,\|Z\|_{L_{\psi_2}(\tau_{N,d})}$.}

For every $z\in\mathbb{C}^N$, 
\begin{equation*}
d!\sum_{|A|=d}|z_A|^2  \le  \left(\sum_{j=1}^N|z_j|^2\right)^d, 
\end{equation*}
because each square-free product occurs exactly $d!$ times in the expansion
on the right. Hence 
\begin{equation*}
\left(\sum_{|A|=d}|z_A|^2\right)^{1/2}  \le  \frac1{\sqrt{d!}}\left\|
z\right\|_2^d. 
\end{equation*}
For $p\ge2$, 
\begin{equation*}
\left\| z\right\|_2\le N^{1/2-1/p}\left\| z\right\|_p. 
\end{equation*}
Therefore 
\begin{equation*}
\lastrev{\int_{S_{N,d}}Z(\varepsilon)\,d\tau_{N,d}(\varepsilon)
\le C\sqrt{\frac{(1+\log d)\log2}{d!}}\,N^{(d+1)/2-d/p}
=K_dN^{(d+1)/2-d/p},}
\end{equation*}
\lastrev{where $K_d:=C\sqrt{(1+\log d)\log2/d!}$.} At least one
choice of signs satisfies \eqref{eq:KSZ}.
\end{proof}

\begin{theorem}
\label{thm:optimal} Let $d\ge1$, $d<p\le\infty$, and let $\Phi$ be any finite local system. If $t>0$ and there is a constant $C$, independent of $N
$, such that 
\begin{equation}  \label{eq:r-hypothesis}
\left(\sum_{s(\alpha)=d}|\widehat f_\Phi(\alpha)|^t\right)^{1/t} \le
C\left\| f\right\|_{p,\Phi}
\end{equation}
for every $N\ge d$ and every $f\in\mathcal{V}_d(\Phi,N)$, then 
\begin{equation*}
t\ge\rho_{d,p}. 
\end{equation*}
Thus the exponent in Theorem C is optimal throughout the full
Hardy--Littlewood range $d<p\le\infty$ for every finite local system.
\end{theorem}

\begin{proof}
Fix $u_0$ and $M_{u_0}>0$ as above.
Suppose first that $d<p\le2d$. Apply \eqref{eq:r-hypothesis} to the function 
$f_M$ from Lemma~\ref{lem:disjoint}. It has exactly $M$ nonzero
coefficients, all equal to one. Hence 
\begin{equation*}
M^{1/t}  \le C M_{u_0}^d d^{-d/p}M^{1-d/p}. 
\end{equation*}
Since $M_{u_0}^d$ is fixed as $M\to\infty$, the comparison gives 
\begin{equation*}
\frac1t\le1-\frac dp, 
\end{equation*}
and therefore 
\begin{equation*}
t\ge\frac{p}{p-d}=\rho_{d,p}. 
\end{equation*}

Now suppose that $p\ge2d$. Take the square-free polynomial in Lemma~\ref%
{lem:squarefree} with the signs from Lemma~\ref{lem:KSZ}. It has exactly $%
\binom{N}{d}$ nonzero coefficients, all of modulus one. Hence %
\eqref{eq:r-hypothesis} and \eqref{eq:squarefreeidentity} imply 
\begin{equation*}
\binom{N}{d}^{1/t}  \le C M_{u_0}^d K_dN^{(d+1)/2-d/p}. 
\end{equation*}
For $N\ge2d$, 
\begin{equation*}
\binom{N}{d}  =\frac{N(N-1)\cdots(N-d+1)}{d!}  \ge\frac{(N/2)^d}{d!}. 
\end{equation*}
Therefore 
\begin{equation*}
N^{d/t-(d+1)/2+d/p}  \le C M_{u_0}^d K_d2^{d/t}(d!)^{1/t}. 
\end{equation*}
Letting $N\to\infty$ forces 
\begin{equation*}
\frac d{t}\le\frac{d+1}{2}-\frac dp, 
\end{equation*}
which, for finite $p$, is equivalent to 
\begin{equation*}
t\ge\frac{2dp}{dp+p-2d}=\rho_{d,p}. 
\end{equation*}
At $p=\infty$, the same comparison gives $d/t\le(d+1)/2$, hence $t\ge
p_d=\rho_{d,\infty}$. At $p=2d$ both arguments give the same lower bound $%
t\ge2$.
\end{proof}

Together with the upper estimate above, this completes the proof of Theorem~C.

\begin{remark}
\label{rem:Tao-system} An example beyond cyclic character systems is
furnished by Tao's order-six complex Hadamard matrix $H$; see \cite[Section~5.6.4]%
{TadejZyczkowski2006}. Its entries have modulus one, its first column is
constant, and $H^*H=6I_6$. Its columns therefore form a local orthonormal
basis on a uniform six-point set. This matrix is not equivalent to the
order-six Fourier matrix. Theorems~A, B and C apply to the resulting system.
\end{remark}

\medskip\noindent\textbf{Funding.} D. N\'u\~nez-Alarc\'on, D. Pellegrino,
and A. Raposo Jr. were supported by CNPq Grant 406457/2023-9
(CNPq/MCTI Call No.~10/2023) and CNPq Grant 403964/2024-5
(MCTI/CNPq Call No.~16/2024). In addition, D. Pellegrino was supported by
CNPq Grant 305807/2025-0 and A. Raposo Jr. by CNPq Grant 302341/2025-0.

\medskip\noindent\textbf{Competing interests.} The authors declare that they
have no competing interests.

\medskip\noindent\textbf{Data availability.} No datasets were generated or
analyzed in this work.

\medskip\noindent\textbf{Declaration of generative AI and AI-assisted technologies.}
During the preparation of this manuscript, the authors used OpenAI's ChatGPT as a generative-AI tool for exploratory calculations, consistency checks, organization of arguments, and drafting and editorial assistance. All AI-assisted material retained in the manuscript was independently reviewed and verified by the authors. The manuscript was written and edited by the authors, who take full responsibility for the content of the work.

\end{document}